\documentclass[11pt,a4paper,reqno,oneside]{amsart}

\usepackage{cite}

\usepackage{color}

\definecolor{gr}{rgb}   {0.,   0.69,   0.23 }
\definecolor{bl}{rgb}   {0.,   0.5,   1. }
\definecolor{mg}{rgb}   {0.85,  0.,    0.85}
\definecolor{or}{rgb}   {0.9,  0.5,   0.}

\newcommand{\Gr}{}
\newcommand{\Bk}{}

\newcommand{\dd}{\mathrm{d}}
\renewcommand{\Tilde}{\widetilde}

\newcommand{\spec}{\mathop{\mathrm{spec}}}

\newcommand{\ZZ}{\mathbb{Z}}
\newcommand{\RR}{\mathbb{R}}

\newcommand{\NN}{\mathbb{N}}
\newcommand{\TT}{\mathbb{T}}
\newcommand{\cO}{\mathcal{O}}
\newcommand{\cD}{\mathcal{D}}
\newcommand{\cH}{\mathcal{H}}

\newcommand{\cE}{\mathcal{E}}

\newcommand{\mx}{\mathrm{max}}

\newtheorem{theorem}{Theorem}[section]
\newtheorem{prop}[theorem]{Proposition}
\newtheorem{definition}[theorem]{Definition}
\newtheorem{cor}[theorem]{Corollary}
\newtheorem{lemma}[theorem]{Lemma}

\numberwithin{equation}{section}

\theoremstyle{definition}

\newtheorem{rem}[theorem]{Remark}

\begin{document}

\title[]{An effective Hamiltonian for
the eigenvalue asymptotics of the Robin Laplacian with a large parameter}

\author{Konstantin Pankrashkin}

\address{Laboratoire de math\'ematiques (UMR 8628 du CNRS), Universit\'e Paris-Sud, B\^atiment 425, 91405 Orsay Cedex, France}

\email{konstantin.pankrashkin@math.u-psud.fr}
\urladdr{http://www.math.u-psud.fr/~pankrash/}

\author{Nicolas Popoff}

\address{Institut math\'ematique de Bordeaux, Universit\'e Bordeaux 1,
351 cours de la lib\'eration, 33405 Talence Cedex, France }

\email{nicolas.popoff@math.u-bordeaux1.fr}
\urladdr{http://www.math.u-bordeaux1.fr/~npopoff/}

\begin{abstract}
We consider the Laplacian on a class of smooth domains $\Omega\subset \RR^{\nu}$, $\nu\ge 2$,
with attractive Robin boundary conditions:
\[
Q^\Omega_\alpha u=-\Delta u, \quad \dfrac{\partial u}{\partial n}=\alpha u  \text{ on } \partial\Omega, \ \alpha>0,
\]
where $n$ is the outer unit normal, and study the asymptotics of its eigenvalues $E_{j}(Q^\Omega_\alpha)$
as well as some other spectral properties for $\alpha\to+\infty$
We work with both compact domains and non-compact ones with a suitable behavior at infinity.
For domains with compact $C^2$ boundaries and fixed $j$, we show that
\[
E_{j}(Q^\Omega_\alpha)=-\alpha^2+\mu_j(\alpha)+{\mathcal O}(\log \alpha),
\]
where  $\mu_j(\alpha)$ is the $j^{\mbox{th}}$ eigenvalue, as soon as it exists,
of $-\Delta_{S}-(\nu-1)\alpha H$ with $(-\Delta_{S})$ and $H$ being respectively
the positive Laplace-Beltrami operator and the mean curvature on $\partial\Omega$.
Analogous results are obtained for a class of domains with non-compact
boundaries. In particular, we discuss the existence of eigenvalues
in non-compact domains and the existence of
spectral gaps for periodic domains.
We also show that
the remainder estimate can be improved under stronger regularity assumptions.

The effective Hamiltonian $-\Delta_{S}-(\nu-1)\alpha H$ enters
the framework of semi-classical Schr\"odinger operators on manifolds, and we provide
the asymptotics of its eigenvalues in the limit $\alpha\to+\infty$ under various
 geometrical assumptions.
In particular, we describe several cases for which our asymptotics provides gaps
between the eigenvalues of $Q^\Omega_\alpha$ for large $\alpha$.

\medskip

\centerline{\small{\it 2010 Mathematics Subject Classification}: 35P15, 35J05, 49R05, 58C40.}
\end{abstract}

\maketitle

\section{Introduction}

Let $\Omega\subset \RR^\nu$, $\nu\ge 2$, be an open set with a sufficiently regular
boundary $S:=\partial \Omega$. 
For $\alpha\in\RR$, denote by $Q^\Omega_\alpha$ the operator
$Q^\Omega_\alpha u =-\Delta u$ on the functions $u$ defined in $\Omega$
and satisfying the Robin boundary condition
\[
\dfrac{\partial u} {\partial n} =\alpha u \text{ on } S,
\]
where $n$ is the outer unit normal at $S$. More precisely, $Q^\Omega_\alpha$ is
the self-adjoint operator in $L^2(\Omega)$ associated with
the quadratic form $q^\Omega_\alpha$ defined on the domain
$\cD(q^\Omega_\alpha)=H^1(\Omega)$ by
\[
q^\Omega_\alpha(u,u)=\int_{\Omega} |\nabla u|^2\dd x -\alpha \int_S u^2\dd S,
\]
where $\dd S$ stands for the $(\nu-1)$-dimensional Hausdorff measure on $S$, which is closed and semibounded from below
under suitable assumptions (e.g. if $S$ is compact or with a suitable behavior at infinity, see below),
and we denote by $E_j^\Omega(Q^\Omega_\alpha)$ the $j^{\mbox{th}}$ eigenvalue of $Q^\Omega_\alpha$ below the bottom
of the essential spectrum, as soon as it exists. The aim of the paper is to obtain
new results on the asymptotics of the eigenvalues as $\alpha$ tends to $+\infty$.

The problem appears in various applications, such as reaction-diffusion processes \cite{LaOckSa98}
and the enhanced surface superconductivity \cite{GiSm07}, and the
related questions were already discussed in the previous works by various authors.
Let us present briefly the state of art for compact domains.
It was shown in \cite{LaOckSa98,LevPar08} that for piecewise smooth Liptschotz domain
one has $E_1(Q^\Omega_\alpha)=-C_\Omega \alpha^2+o(\alpha^2)$
as $\alpha\to+\infty$, where $C_\Omega\ge 1$ is a constant depending on the geometric properties of $\Omega$.
In particular, $C_\Omega=1$ for $C^1$ domains, see \cite{LouZhu04,DaKe10}. More detailed asymptotic expansions
for some specific non-smooth domains were considered in \cite{LevPar08,HP,kp14}.
As for smooth domains, a more detailed result was obtained first in \cite{Pank13,ExMinPar14} for $\nu=2$
and then in \cite{pp} for any $\nu\le 2$: if the domain is $C^3$ and $j\in\NN$ is fixed, then
\begin{equation}
      \label{eq-hhh}
E_{j}(Q^\Omega_\alpha)=-\alpha^2-(\nu-1) H_{\max}\alpha+\cO(\alpha^{2/3}),
\end{equation}
where $H_{\max}$ is the maximum of the mean curvature $H$ of the boundary (the exact definition
will be recalled below). Remark that this asymptotics together with isompermetric inequalities
for the mean curvature have played an important role
for the so-called reverse Faber-Krahn inequality, see \cite{FreKre14,pp}.
The result \eqref{eq-hhh} was also obtained in \cite{em} for a class of non-compact planar domains.

Although the asymptotics \eqref{eq-hhh} shows the influence of the geometry on first orders,
it is not sufficient to distinguish the influence of the number $j$ of the eigenvalue, and
to estimate the gap between eigenvalues. 
For $\nu=2$, a complete asymptotic expansion of the eigenvalues of the form 
\begin{equation}
    \label{eq-HK}
E_{j}(Q_{\alpha}^{\Omega}) = -\alpha^2-H_{\max}\alpha+(2j-1)\sqrt{\frac{\big|H''(s_{0})\big|}{2}}\,\alpha^{1/2}+\sum_{k=0}^{N}\gamma_{j,k}\alpha^{-\frac{j}{2}}+o(\alpha^{-\frac{N}{2}}),
\quad \gamma_{j,k}\in \RR,
\end{equation}
where $j\in \NN$ is arbitrary but fixed,
was proved in \cite{HK} proved  under the assumption that the curvature $s\mapsto H(s)$ admits a unique non-degenerated maximum at $s_{0}$ and the second derivative
is taken with respect to the arc-length. Such a hypothesis is reminiscent of several works about the first eigenvalues of the magnetic Laplacian in the semi-classical limit, see \cite{HeMo01}, which involves the localization of the eigenfunctions at the boundary and allow expansion of the associated eigenvalues. More precisely, such a maximum of the curvature acts as a potential well for Schr\"odinger operators 
in the harmonic approximation, see \cite{DiSj99}.

In order to describe our results, let us introduce the necessary notation and the class of domains we consider in this article.
Recall that $s\mapsto n(s)$ is the Gauss map on $S$, i.e. $n(s)$ is the outward pointing unit normal vector at $s\in S$.
Consider the shape operator $L_s$ at $s\in S$, which is defined by
$L_s:=\dd n(s): T_s S\to T_s S$,
and let $\kappa_1(s),\dots,\kappa_{\nu-1}(s)$ be its eigenvalues, called the principal curvatures.
Our results will be valid for the so-called \emph{$C^k$-admissible} domains defined as follows:

\begin{definition}
\label{D:admissible}
Let $k\geq2$. A domain $\Omega\subset \RR^{\nu}$ is called \emph{$C^{k}$-admissible},
if its boundary is $C^{k}$, and, in addition, the following holds: 
\begin{itemize}
\item[(H1)] There exists $\delta>0$ such that the map $\Phi$ defined by
\begin{equation}
      \label{D:TubNei}
\Sigma:=S\times (0,\delta)\ni (s,t)\mapsto \Phi(s,t):=s-tn(s)\in \Phi\big(S\times (0,\delta)\big)
\end{equation}
is a diffeomorphism and \Gr its image is contained in $\Omega$.
\item[(H2)] The curvatures $s\mapsto \kappa_{i}(s)$ are in $L^{\infty}(S)$. Moreover, if $k\geq3$, their gradients are also bounded.
\end{itemize}
\end{definition}
The assumption (H1) is quite standard if one deals with non-compact domains, and it is
sometimes called the \emph{non-overlap condition}, cf. e.g.~\cite{CarExKr04}.
We remark that any $C^k$ domain with a compact boundary is $C^k$-admissible.

In what follows we denote by $K(s)$ the sum of the principal curvatures:
\[
K(s)=\kappa_1(s)+\dots+\kappa_{\nu-1}(s)\equiv \mathop{\mathrm{tr}} L_s.
\]
Remark that the quantity $H:=K/(\nu-1)$  is exactly the mean curvature on $S$. We will denote
\[
K_\mx:=\sup_{s\in S} K(s).
\]
Note that the operator $Q^\Omega_\alpha$ can have a non-empty essential spectrum,
and it is more convenient to work with the Rayleigh quotients instead of the eigenvalues.
For this purpose, recall the min-max principle for the eigenvalues of self-adjoint
operators, see e.g.~\cite[Sec.~4.5]{davies}:
Let $Q$ be a lower semibounded  self-adjoint operator in a Hilbert space $\cH$ and $q$ be its quadratic form.
Denote 
\begin{align}
E(Q)&:=\inf\spec\nolimits_\text{ess} Q \quad \text{(we use the convention $\inf\emptyset=+\infty$)},\nonumber\\
   \label{E:QR}
E_j(Q)&:=\min_{\substack{L\subset \cD(q),\\ \dim L=j}} \max_{\substack{u\in L,\\ u\ne 0}}
\dfrac{q(u,u)}{\langle u, u\rangle}, \quad j\in \NN,
\end{align}
then:
\begin{itemize}
\item if $E_j(Q)<E(Q)$, then $E_j(Q)$ is the $j^{\mbox{th}}$ eigenvalue of $Q$,
\item if $E_j(Q)\ge E(Q)$, then $E_k(Q)=E(Q)$ for all $k\ge j$.
\end{itemize}

Our main result gives a comparison between the Rayleigh quotients of $Q^\Omega_\alpha$
with those of an auxiliary Schr\"odinger operator acting on the boundary and in which
the Robin coefficient $\alpha$ appears as a coupling constant:
\begin{theorem}\label{thm1}
Let $\nu\geq2$ and $\Omega\subset\RR^\nu$ be a $C^2$-admissible domain. Furthermore,
let $-\Delta_S$ denote the positive Laplace-Beltrami operator on $S$ viewed as
a self-adjoint operator in $L^2(S,dS)$. For any fixed $j\in \NN$ one has
\begin{equation}
       \label{eq-th1}
E_j(Q^\Omega_\alpha)=-\alpha^2 + E_j(-\Delta_S-\alpha K) +\cO(\log\alpha),
\quad \alpha\to+\infty.
\end{equation}
\end{theorem}

The proof is presented in Sections \ref{th1pr1} and \ref{SS:lb1}.
Furthermore, in Section \ref{th2pr} we show that the remainder
estimate can be improved under additional assumptions:

\begin{theorem}\label{thm2}
Let the assumptions of Theorem~\ref{thm1} hold. In addition, assume that
$\Omega$ is $C^{3}$-admissible and that $K$ reaches its maximum, then for any fixed $j\in\NN$ one has
\begin{equation}
       \label{eq-th1a}
E_j(Q^\Omega_\alpha)=-\alpha^2 + E_j(-\Delta_S-\alpha K) +\cO(1), \quad \alpha\to+\infty.
\end{equation} 
\end{theorem}

Let us emphasize on the fact that no assumptions are done on the behavior of $K$ near the set
$K^{-1}(K_\text{max})$ for the above results. Using various methods available for the study
of the effective Hamiltonian $-\Delta_S-\alpha K$ one can deduce more precise asymptotics
under various assumptions, which improve the results of preceding works
of weaken the respective assumptions. In particular, as a generalization of \eqref{eq-hhh}
in Section~\ref{secor} we obtain:
\begin{cor}\label{cor1}
Let $\Omega\subset \RR^\nu$ be a $C^{2}$-admissible domain, then for each fixed $j\in\NN$
we have $E_j(Q^\Omega_\alpha)=-\alpha^2-K_\mx \alpha +o(\alpha)$.
\end{cor}
\begin{rem}
If the boundary $S$ is compact, then either $\spec_\text{ess} Q^\Omega_\alpha=\emptyset$ (if $\Omega$ is bounded)
or $\spec_\text{ess} Q^\Omega_\alpha=[0,+\infty)$ (if $\Omega$ is unbounded). In the latter case,
it is standard to check that $Q^\Omega_\alpha$ has at most finitely many negative eigenvalues.
On the other hand, for any fixed $j$ one has $E_j(Q^\Omega_\alpha)<0$ if $\alpha$ is sufficiently large,
in particular, $E_j(Q^\Omega_\alpha)<E(Q^\Omega_\alpha)$. 
Therefore, by the min-max principle, each $E_j^\Omega(\alpha)$ is an eigenvalue of $Q^\Omega_\alpha$
if $\alpha$ is sufficiently large.

The preceding observation does not hold for domain with non-compact boundaries.
In particular, in \cite{em} one can find various examples of domains $\Omega$
with curved non-compact boundaries such that the respective operators $Q^\Omega_\alpha$
have a purely essential spectrum for any  $\alpha>0$.
Nevertheless, the existence of eigenvalues can be guaranteed
by an additional assumption:
\begin{equation}
\label{E:H3}
\mbox{(H3): }\quad K_\infty :=\limsup_{s\to \infty} K(s) < K_{\max},
\end{equation}
which allows one to prove the following result extending several estimates of~\cite{em}:
\begin{cor}\label{cor1a}
Let $\Omega\subset\RR^\nu$ be $C^k$-admissible with non-compact boundary and satisfy \eqref{E:H3}, then for any 
$N\in \NN$ there exists $\alpha_N>0$ such that for $\alpha>\alpha_{N}$ the operator $Q^\Omega_\alpha$
has at least $N$ eigenvalues below the essential spectrum. The behavior of the $j^{\mbox{th}}$ eigenvalue $E^\Omega_j(\alpha)$
with a fixed $j$ is given by \eqref{eq-th1} or, if $\Omega$ is $C^3$, by \eqref{eq-th1a}.
\end{cor}
The proof is given in Section~\ref{secor}.
\end{rem}

More detailed asymptotic expansions for the eigenvalues can be deduced
by using the toolbox of the semi-classical analysis of Schr\"odinger operators on manifolds,
where the mean curvature acts as a potential. In Section \ref{S:semiclassical} we describe the results
involved by standard hypotheses on the potential $K$. 
In particular, the following hold:
\begin{cor}
\label{C:nondegenerate}
Let $\Omega\subset \RR^\nu$ be a $C^5$-admissible domain. If $\partial\Omega$ is non-compact, assume \eqref{E:H3}.
Furthermore, assume that $K$ admits a unique global maximum at $s_{0}$ and that the Hessian of $(-K)$
at $s_{0}$ is positive-definite. Denote by $\mu_{k}$ its eigenvalues
and set
\[
\cE=\bigg\{ \sum_{k=1}^{\nu-1}\sqrt{\frac{\mu_{k}}{2}}\,\big(2n_{k}-1\big), n_{k}\in \NN\bigg\}.
\]
Then for each $j\in\NN$ there holds, as $\alpha\to+\infty$:
$$E_{j}(Q^\Omega_\alpha)=-\alpha^2-K_{\max}\alpha+e_{j}\alpha^{1/2}+\cO(\alpha^{1/4}),$$
where $e_{j}$ is the $j^{\mbox{th}}$ element of $\cE$, counted with multiplicity.
Moreover, if $e_{j}$ is of multiplicity one, the remainder estimate can be improved to $\cO(1)$. 
\end{cor}

\begin{cor} 
\label{C:degenerate}
Let $\Omega\subset \RR^{2}$ be a $C^{2p+3}$-admissible domain with some integer $p>1$, and
assume \eqref{E:H3} if $\partial\Omega$ is unbounded.
Assume that the curvature of the boundary admits a unique global maximum at $s_{0}$,
which is degenerated in the following sense: 
\[
K(s)=K(s_{0})-C_{p}(s-s_{0})^{2p}+\cO\big((s-s_{0})^{2p+1}\big), \quad C_{p}>0
\]
where $s$ denotes the arc length of the connected component $\Gamma$ of the boundary where $K$ is maximal, then
for each $j\in\NN$ there holds, as $\alpha\to+\infty$:
\[
E_{j}(-\Delta_{S}- \alpha K)=-K_{\max}\alpha+e_{j}\alpha^{\frac{1}{p+1}}+\cO\big(\alpha^{\frac{1}{2(p+1)}}\big),
\]
where $e_{j}$ is the $j^{\mbox{th}}$ eigenvalue of the operator $-\partial_{s}^2+C_{p}s^{2p}$
acting on $L^2(\RR)$.
If $\partial \Omega$ is $C^{2p+4}$ smooth, then the remainder can be replaced by $\cO(1)$.
 \end{cor}

Finally, in Section \ref{secper} we consider the case when $\Omega$ is periodic with a compact elementary cell.
In that case, the above main results show that the spectral bands of $Q^\Omega_\alpha$
are determined, up to a error term, by the spectral bands of the periodic operator $-\Delta_S-\alpha K$.
In particular, we prove some sufficient conditions guaranteeing the existence
of gaps in the spectrum of $Q^\Omega_\alpha$.

The machinery used for the proof of the main results is quite different from all the previous papers
on the Robin eigenvalues and is based on a detailed analysis of the quadratic form and appears to
be ideologically very close to the one for the Laplacians in thin domains, cf. \cite{FS1,FS2,kr09,kr}.
The reduced operator $-\Delta_S-\alpha K$ appeared
already in~\cite{kr} in the study of suitable Laplacians in thin neighbordnood of hypersurfaces,
and the results from Section \ref{S:semiclassical} provide improvements
of the asymptotics given in \cite[Theorem 1.1]{kr},
under the respective geometric assumptions.

\section{Auxiliary estimates}

We remark first that, as we deal with real-valued operators only,  we will
work everywhere with real Hilbert spaces. Let us prove some technical estimates  which will be used in the proof
of the main results.

\begin{lemma}\label{lemd}
For $\alpha>0$ and $\delta>0$, denote by $T^{D}$ the operator $f\mapsto -f''$
acting in $L^2(0,\delta)$ on the domain
\[
\cD(T^{D})=\big\{
f\in H^2(0,\delta): f'(0)=-\alpha f(0),\, f(\delta)=0
\}.
\]
Then, as $\delta\alpha$ tends to $+\infty$, the operator $T^{D}$
has a unique negative eigenvalue $E^{D}$, which satisfies
\begin{equation}
\label{E:estimeED}
E^{D}=-\alpha^2 + \cO(\alpha^2e^{-\delta\alpha}).
\end{equation}
Furthermore, if $\psi^{D}$ is an associated  normalized eigenfunction, then
\[
\psi^{D}(0)^2=2\alpha+ \cO(\alpha e^{-\delta\alpha}).
\]
\end{lemma}

\begin{proof}
The assertion was partially proven in Lemma~A.2 of \cite{HP} by direct computations: it was shown
that the operator $T^{D}$ has a unique negative eigenvalue, that
$E^{D}=-k^2$ with $k=\alpha+\cO(\alpha e^{-\delta\alpha})$,
and, finally, that $\psi^{D}(t)=C (e^{k(t-\delta)}-e^{-k(t-\delta)})$, where
$C$ is a normalizing constant. We have then
\[
1=\|\psi^{D}\|^2_{L^2(0,\delta)}=C^2\Big(
\dfrac{e^{2\delta k}-e^{-2\delta k}}{2k}-2\delta
\Big),
\quad
C^2=\dfrac{2ke^{-2\delta k}}{1-4\delta k e^{-2\delta k}-e^{-4\delta k}},
\]
which gives
\[
\psi^{D}(0)^2=2k \dfrac{(1-e^{-2\delta k})^2}{1-4\delta k e^{-2\delta k}-e^{-4\delta k}}
=2k+ \cO(\delta k^2 e^{-2\delta k})=2\alpha+\cO(\alpha e^{-\delta \alpha}). \qedhere
\]
\end{proof}

\begin{lemma}\label{lemn}
Let $\beta\geq0$ be fixed. For $\alpha>0$ and $\delta>0$, denote by $T^{\beta}$ the operator $f\mapsto -f''$
acting in $L^2(0,\delta)$ on the domain
\[
\cD(T^{\beta})=\big\{
f\in H^2(0,\delta): f'(0)=-\alpha f(0),\, f'(\delta)=\beta f(\delta)
\}.
\]
Then, as $\delta\alpha$ tends to $+\infty$, the operator $T^{\beta}$ has a unique negative eigenvalue $E^{\beta}$, which satisfies
\begin{equation}
\label{E:asymptoEbeta}
E^{\beta}=-\alpha^2 + \cO(\alpha^2e^{-\delta\alpha}).
\end{equation}
Furthermore, if $\psi^{\beta}$ is an associated normalized eigenfunction,
then
\begin{align}
        \label{eq-fn1}
\psi^{\beta}(0)^2&=2\alpha+ \cO(\alpha e^{-\delta\alpha}),\\
        \label{eq-fn1d}
\psi^{\beta}(\delta)^2&=4\alpha e^{-2\delta\alpha}+ \cO(\alpha e^{-3\delta\alpha}),\\
        \label{eq-fn2}
 \|(\psi^{\beta})'\|^2_{L^2(0,\delta)}&=\alpha^2+ \cO(\alpha^2 e^{-\delta\alpha}).
\end{align}
In addition,
 \begin{equation}
       \label{eq-fpsi}
       \|f'\|^2_{L^2(0,\delta)}-\alpha f(0)^2-\beta  f(\delta)^2\ge 0
      \text{  for any $f\in H^1(0,\delta)$ with $\langle f ,\psi^{\beta}\rangle_{L^2(0,\delta)}=0$.}
 \end{equation}
\end{lemma}
\begin{proof}
Once again $E^{\beta}$ is clearly negative, and we denote by $k$ the positive number such that $E^{\beta}=-k^2$, so that 
\[
\psi^{\beta}(t)=C\bigg(\Big(1+\dfrac{\beta}{k}\Big)e^{k(t-\delta)}+\Big(1-\dfrac{\beta}{k}\Big)e^{-k(t-\delta)}\bigg),
\]
with $C$ a normalizing constant. Then the condition $\psi'(0)=-\alpha\psi(0)$ is equivalent to 
$$ \delta k\frac{\sinh(\delta k)-\dfrac{\beta}{k}\cosh(\delta k)}{\cosh(\delta k)-\dfrac{\beta}{k}\sinh(\delta k)}=\delta\alpha.$$
Following literally the proof of \cite[Lemma A.1]{HP} treating the case $\beta=0$, we get the existence of a unique solution as $\alpha$ gets large, which satisfies $k=\alpha+\cO(\alpha e^{-\delta\alpha})$, which gives
the asymptotics of $E^{\beta}=-k^2$. Moreover, the other eigenvalues of $T^{\beta}$ are positive, and since the quadratic form for $T^{\beta}$ is
\[
t^{\beta}(f,f)=\|f'\|^2_{L^2(0,\delta)}-\alpha f(0)^2-\beta f(\delta)^2, \quad \cD(t^{\beta})=H^1(0,\delta),
\]
the assertion \eqref{eq-fpsi} follows from the spectral theorem for self-adjoint operators.
Then \eqref{eq-fn1} and \eqref{eq-fn1d} are
obtained as in the the proof of Lemma~\ref{lemd}.
Finally, substituting this estimate into the equality $t^{\beta}(\psi,\psi)=E^{\beta}$ we obtain~\eqref{eq-fn2}.
\end{proof}

Finally, we will need a suitable form of the Sobolev inequality, see e.g.
Lemma~8 in~\cite{kuch}:
\begin{lemma}\label{lemsob}
For any $0<\ell\le a$ and $f\in H^1(0,a)$ there holds, with $\xi\in\{0,a\}$,
\[
f(\xi)^2\le \ell \int_0^a f'(t)^2dt+ \dfrac{2}{\ell}\int_0^a f(t)^2dt.
\]
\end{lemma}

\section{Reduction to the analysis near the boundary}\label{sec1}

\subsection{Dirichlet-Neumann bracketing}\label{dtnbra}
The first steps of the analysis are essentially the same as in \cite{pp}. For $\delta>0$ denote
\[
\Omega_\delta:=\big\{
x\in \Omega: \, \inf_{s\in S} |x-s|<\delta
\},
\quad
\Theta_\delta:=\Omega\setminus \overline{\Omega_\delta},
\]
and let $q^{\Omega,N,\delta}_\alpha$ and $q^{\Omega,D,\delta}_\alpha$
be the quadratic forms given by the same expression as $q^\Omega_\alpha$
but acting on the domains
\begin{gather*}
\cD(q^{\Omega,N,\delta}_\alpha)=H^1(\Omega_\delta)\oplus H^1(\Theta_\delta),
\quad
\cD(q^{\Omega,D,\delta}_\alpha)=\Tilde H^1_0(\Omega_\delta)\oplus H^1_0(\Theta_\delta),\\
\widetilde H^1_0(\Omega_\delta):=\{f\in H^1(\Omega_\delta): \, f=0 \text{ at } \partial\Omega_\delta\setminus S \},
\end{gather*}
and denote by $Q^{\Omega,N,\delta}_\alpha$ and $Q^{\Omega,D,\delta}_\alpha$
the associated self-adjoint operators in $L^2(\Omega)$.
The inclusions $\cD(q^{\Omega,D,\delta}_\alpha)\subset\cD(q^{\Omega}_\alpha)\subset\cD(q^{\Omega,N,\delta}_\alpha)$
and the min-max principle imply, for each $j\in\NN$, the inequalities
\[
E_j(Q^{\Omega,N,\delta}_\alpha)\le E_j(Q^\Omega_\alpha)\le E_j(Q^{\Omega,D,\delta}_\alpha).
\]
Furthermore, $Q^{\Omega,N,\delta}_\alpha=B^{\Omega,N,\delta}_\alpha \oplus
(-\Delta)^N_{\Theta_\delta}$ and $Q^{\Omega,D,\delta}_\alpha=B^{\Omega,D,\delta}_\alpha \oplus
(-\Delta)^D_{\Theta_\delta}$,
where $B^{\Omega,N,\delta}_\alpha$ and $B^{\Omega,D,\delta}_\alpha$
are the self-adjoint operators in $L^2(\Omega_\delta)$ associated respectively with the quadratic forms
\begin{gather*}
b^{\Omega,\star,\delta}_\alpha(u,u)=\int_{\Omega_\delta} |\nabla u|^2dx -\alpha \int_S u^2d S,
\quad \star\in\{N,D\},\\
 \cD(b^{\Omega,N,\delta}_\alpha)=H^1(\Omega_\delta),\quad
\cD(b^{\Omega,D,\delta}_\alpha)=\widetilde H^1_0(\Omega_\delta),
\end{gather*}
and  $(-\Delta)^N_{\Theta_\delta}$ and $(-\Delta)^D_{\Theta_\delta}$
denote the Neumann and the Dirichlet Laplacian in $\Theta_\delta$, respectively.
As both Neumann and Dirichlet Laplacians are non-negative, we have the inequalities
\begin{equation}
       \label{eq-est1}
E_j(B^{\Omega,N,\delta}_\alpha)\le E_j(Q^\Omega_\alpha)\le E_j(B^{\Omega,D,\delta}_\alpha)
\text{ for all $j$ with } E_j(B^{\Omega,D,\delta}_\alpha)<0.
\end{equation}
The preceding inequalities are valid for any value of $\delta>0$, but for the rest of the paper we assume that
$\delta$ depends on $\alpha$ in a special way:
\begin{equation}
        \label{eq-dada}
\text{ the value of $\delta$ tends to $0$ and the value of $\delta\alpha$ tends to $+\infty$ as $\alpha$ tends to $+\infty$,}
\end{equation}
and the  precise dependence will be chosen later.

\subsection{Change of variables}
\label{S:CV}
In order to study the eigenvalues of the operators $B^{\Omega,N,\delta}_\alpha$
and $B^{\Omega,D,\delta}_\alpha$ we proceed first with a change of variables in $\Omega_\delta$
with small $\delta$. The computations below are very similar to those
performed in~\cite{CarExKr04} for a different problem.

By assumption, for $\delta>0$ sufficiently small,
 the map $\Phi$ defined in \eqref{D:TubNei}
is a diffeomorphism between $\Sigma$ and $\Omega_{\delta}$. 
The metric $G$ on  $\Sigma$ induced by this embedding is
\begin{equation}
      \label{eq-gg}
G=g\circ (I_s-tL_s)^2 + \dd t^2,
\end{equation}
where $I_s:T_s S\to T_s S$ is the identity map,
and $g$ is the metric on $S$ induced by the embedding in $\RR^\nu$.
The associated volume form $\dd\Sigma$ on $\Sigma$ is
\begin{equation}
      \label{E:metrics}
\dd\Sigma =|\det G|^{1/2}\dd s\, \dd t=\varphi(s,t)|\det g|^{1/2} \dd s \,\dd t=\varphi \, \dd S \,\dd t,
\end{equation}
where
\[
\dd S=|\det g|^{1/2}\dd s
\]
is the induced $(\nu-1)$-dimensional volume form on $S$, and the weight $\varphi$ is given by
\begin{equation}
      \label{eq-r1}
\varphi(s,t):=\big|\det (I_s-t L_s)\big|=1- t \mathop{\mathrm{tr}}L_s +p(s,t)t^2\equiv
1-K(s) t +p(s,t)t^2,
\end{equation}
with $p$ being a polynomial in $t$ with coefficients which are bounded and continuous in $s$.  In particular,
\begin{equation}
\label{E:DLtvarphi}
|\varphi(s,t)-1| \leq \|\partial_{t}\varphi\|_{\infty}\delta \text{ for all } (s,t)\in \Sigma.
\end{equation}
Let us recall that for a function $f: S \mapsto \RR$, the boundedness of the gradient $\nabla_{s}f$, as stated in  Definition \ref{D:admissible},
is understood for the norm on the tangent spaces: $ \|\nabla_{s}f\|_{\infty}=\sup_{s\in S}\| \nabla_{s}f(s)\|_{T_{s}S}$, with
\begin{align*}
 \| \nabla_{s}f(s)\|_{T_{s}S}^2 &= \sum_{\rho,\mu}g_{\rho\mu}(s) \left(\sum_{k} g^{\rho k}(s)\partial_{k}f(s) \right)\left(\sum_{\ell} g^{\mu\ell}(s)\partial_{\ell}f(s) \right)
 \\
 &=\sum_{\rho,\mu}g^{\rho\mu}(s) \partial_{\rho}f(s)\partial_{\mu}f(s)
 \quad \mbox{with} \ (g^{\rho \mu})=g^{-1}.
  \end{align*}
For future uses, we summarize some obvious properties of $\varphi$:
\begin{lemma}
\label{Linter}
Let $\Omega$ be a $C^{2}$-admissible domain, the for $\delta$ small,
the functions $L_{s}$, $K$ are bounded on $S$, and the functions $\partial_{t}\varphi$, $\partial_{t}^2\varphi$, $\partial_{t}\varphi^{-1/2}$, $(\partial_{t}\varphi^{-1/2})\varphi^{1/2}$, and
$\partial_{t}\big((\partial_{t}\varphi^{-1/2})\varphi^{1/2}\big)$ are bounded on $\Sigma$.
If, in addition, $\Omega$ is $C^{3}$-admissible, then $\nabla_{s}\partial_{t}(\varphi^{-1/2})$ is bounded on $\Sigma$.
\end{lemma}
In particular,  for some $C>0$ we have
\begin{equation}
\label{E:UBps}
\Big|\sum_{\rho,\mu}g^{\rho\mu}(s) \partial_{\rho}\partial_{t}(\varphi^{-1/2})(s,t)\partial_{\mu}\partial_{t}(\varphi^{-1/2})
\Big|\le C, \quad (s,t)\in\Sigma.
\end{equation}

Now consider the unitary map
\[
U: L^2(\Omega_\delta)\to L^2(\Sigma,\dd\Sigma),
\quad
Uf=f\circ \Phi,
\]
where $\Phi$ is the map from \eqref{D:TubNei},
and the quadratic forms 
\[
h^{\star}_\alpha(f,f)=b^{\Omega,\star,\delta}_\alpha(U^{-1}f, U^{-1}f),
\quad \cD (h^\star_\alpha)=U \cD (b^{\Omega,\star,\delta}_\alpha),
\quad \star\in\{N,D\}.
\]
We have then, using the Einstien summation rule for indices,
\begin{align*}
h^{N}_\alpha(u,u)&=\int_{\Sigma} G^{jk}\partial_j u  \,\partial_k u\,\dd\Sigma
-\alpha \int_S |u(s,0)|^2 \dd S, \quad \cD(h^{N}_\alpha)= H^1(\Sigma),\\
h^{D}_\alpha(u,u)&=\text{the restriction of $h^{N}_\alpha$ to }
\cD(h^{D}_\alpha)=\widetilde H^1_{0}(\Sigma)
\end{align*}
with
\[
\widetilde H^1_{0}(\Sigma):=\big\{u\in H^1(\Sigma): u(\cdot,\delta)=0\big\}, \quad
(G^{jk}):=G^{-1}.
\]
Due to \eqref{eq-gg} we can estimate, with some $C_{g}>0$, depending only on $\|L_{s}\|_{\infty}$:
\[
(1-C_{g}\delta) g^{-1} + \dd t^2\le G^{-1}\le (1+C_{g}\delta) g^{-1} + \dd t^2.
\]
Therefore, we have the form inequalities
\begin{equation}
       \label{eq-cc}
h^-_\alpha\le h^{N}_\alpha
 \quad \text{ and} \quad
h^{D}_\alpha\le h^+_\alpha
\end{equation}
with
\[
\begin{aligned}
h^-_\alpha(u,u)&:=(1-C_{g}\delta)\int_{\Sigma} g^{\rho \mu} \,\partial_\rho u \,\partial_\mu u\,\dd\Sigma 
+\int_{\Sigma} |\partial_t u|^2\dd\Sigma-\alpha \int_S u(s,0)^2 \dd S,\\
& \quad \cD(h^-_\alpha)=\cD(h^{N}_\alpha)=H^1(\Sigma),\\
h^+_\alpha(u,u)&:=(1+C_{g}\delta)\int_{\Sigma} g^{\rho \mu} \,\partial_\rho u \,\partial_\mu u\,\dd\Sigma
+\int_{\Sigma} |\partial_t u|^2\dd\Sigma-\alpha \int_S u(s,0)^2 \dd S,\\
& \quad \cD(h^+_\alpha)=\cD(h^{D}_\alpha)=\widetilde H^1_{0}(\Sigma),
\end{aligned}
\]
where, as usually, $(g^{\rho\mu})= g^{-1}$. In particular, if $H^-_\alpha$ and $H^+_\alpha$ are the self-adjoint operators
acting in $L^2(\Sigma,\dd\Sigma)$ and associated with the forms $h^-_\alpha$ and $h^+_\alpha$ respectively, then
it follows from \eqref{eq-est1} and \eqref{eq-cc} that
\begin{equation}
       \label{eq-est2}
E_j(H^-_\alpha)\le E_j(Q^\Omega_\alpha)\le E_j(H^+_\alpha)
\text{ for all $j$ with } E_j(H^+_\alpha)<0.
\end{equation}

\section{Proof of Theorem \ref{thm1}: upper bound}\label{th1pr1}

Recall that the operator $T^{D}$ has been defined in Lemma \ref{lemd}. We have denoted by $E^{D}$ its lowest eigenvalue, and in \emph{this section}
we denote for shortness $\psi:=\psi^{D}$ an associated normalized eigenfunction.
The function $\psi$ will be used to construct test functions for $H_{\alpha}^{+}$. 
\subsection{An estimate for product functions}
Recall that everywhere we assume that $\delta$ is a function of $\alpha$ satisfying \eqref{eq-dada}.
We have the following estimate:
\begin{lemma}
\label{L:estimatetensor}
For $v\in H^1(S)$, consider a function $u$ defined by $u(s,t)=v(s)\psi(t)$, which belongs to $\cD(h^+_\alpha)$.
There exist positive constants $c_{0}^{+}$ and $c_{1}^{+}$ such that, as $\alpha\to+\infty$,
for any $v\in H^1(S)$ there holds
 \begin{multline}
 \dfrac{h^+_\alpha(u,u)}{\|u\|^2_{L^2(\Sigma,\dd\Sigma)}}-E^{D}
 \\
 \le (1+c_{0}^{+}\delta) \dfrac{(1+c_{1}^{+}\delta)
\displaystyle\int_{S} g^{\rho \mu} \,\partial_\rho v \,\partial_\mu v\,\dd s
-\alpha \langle v, K v\rangle_{L^2(S,\dd S)}}{\|v\|^2_{L^2(S,\dd S)}}\\
+\cO(1+\alpha e^{-\delta\alpha}),
\end{multline}
Moreover, the remainder depends only on $\|L_{s}\|_{\infty}$, $\|K\|_{\infty}$, $\|\partial_{t}\varphi\|_{\infty}$, and $\|\partial^2_t\varphi\|_{\infty}$, and it is independent of $v$.
\end{lemma}
\begin{proof}
Through the estimates we denote by $C_j$ various positive constants. Using \eqref{E:DLtvarphi}, a direct evaluation provides
\begin{equation}
        \label{eq-hmin}
\begin{aligned}
h^+_\alpha(u,u)=&(1+C_{g}\delta)\int_{S\times(0,\delta)} \psi(t)^2 g^{\rho \mu} \,\partial_\rho v(s) \,\partial_\mu v(s)\,\varphi(s,t)\dd S\dd t\\
&+\int_{S} v(s)^2 \int_0^\delta \psi'(t)^2\varphi(s,t)\dd t\, \dd S -\alpha \psi(0)^2 \int_S v^2\dd S\\
\le&(1+C_1\delta)\int_{S} \int_0^\delta \psi(t)^2 g^{\rho \mu} \,\partial_\rho v(s) \,\partial_\mu v(s)\dd t \dd S\\
&+\int_{S} v(s)^2 \int_0^\delta \psi'(t)^2\varphi(s,t)\dd t\, \dd S -\alpha \psi(0)^2 \int_S v^2\dd S\\
=&(1+C_1\delta)\int_{S} g^{\rho \mu} \,\partial_\rho v \,\partial_\mu v\,\dd S\\
&+\int_{S} v(s)^2 \int_0^\delta \psi'(t)^2\varphi(s,t)\dd t\, \dd S -\alpha \psi(0)^2 \int_S v^2\dd S.
\end{aligned}
\end{equation}
Moreover, the constant $C_{1}$ depends only on $\|L_{s}\|_{\infty}$ and $\|\partial_{t}\varphi\|_{\infty}$. \Bk Using a repeated integration by parts together with the boundary conditions satisfied by $\psi$, we obtain for all $s\in S$:
\begin{multline}
    \label{eq-if1}
\int_0^\delta \psi'(t)^2\varphi(s,t)\dd t\\
\begin{aligned}
=&\Big[ \psi(t)\psi'(t) \varphi(s,t)\Big]_{t=0}^{t=\delta}
+\int_0^\delta \psi(t)\big(-\psi''(t)\big) \varphi(s,t) \dd t
-\int_0^\delta \psi(t)\psi'(t) \partial_t \varphi(s,t) \dd t\\
=&-\psi(0)\psi'(0)\varphi(s,0)+E^{D} \int_0^\delta \psi(t)^2 \varphi(s,t) \dd t
-\dfrac{1}{2}\int_0^\delta \partial_t \big(\psi(t)^2\big) \partial_t\varphi(s,t)\dd t\\
=&\alpha\psi(0)^2+E^{D} \int_0^\delta \psi(t)^2 \varphi(s,t) \dd t
-\dfrac{1}{2}\bigg(
\Big[\psi(t)^2\partial_t\varphi(s,t)\Big]_{t=0}^{t=\delta} - 
\int_0^\delta  \psi(t)^2 \partial^2_t\varphi(s,t)\dd t
\bigg)\\
=&\alpha\psi(0)^2+E^{D} \int_0^\delta \psi(t)^2 \varphi(s,t) \dd t
-\dfrac{K(s)}{2} \psi(0)^2 +\dfrac 12 \int_0^\delta  \psi(t)^2 \partial^2_{t}\varphi(s,t)\dd t.
\end{aligned}
\end{multline}
The substitution of \eqref{eq-if1} into \eqref{eq-hmin} gives
\begin{multline}
        \label{eq-hmin2}
h^+_\alpha(u,u)\le(1+C_1\delta)\int_{S} g^{\rho \mu} \,\partial_\rho v \,\partial_\mu v\,\dd S\\
+E^{D}\|u\|^2_{L^2(\Sigma,\dd\Sigma)}-\dfrac{\psi(0)^2}{2}  \langle v, K v\rangle_{L^2(S,\dd S)}+
\dfrac 12 \int_S \int_0^\delta  v(s)^2\psi(t)^2 \partial^2_{t}\varphi(s,t)\dd t\, \dd S.
\end{multline}
As the functions $\partial^2_t\varphi$ and $K$
are bounded, we estimate with the help of Lemma~\ref{lemd}:
\begin{multline}
        \label{eq-hmin3}
h^+_\alpha(u,u)-E^{D}\|u\|^2_{L^2(\Sigma,\dd\Sigma)}\\
\leq
(1+C_1\delta)\int_{S} g^{\rho \mu} \,\partial_\rho v \,\partial_\mu v\,\dd S
-\alpha \langle v, K v\rangle_{L^2(S,\dd S)}\\+ \cO(1+\alpha e^{-\delta\alpha})\|v\|^2_{L^2(S,\dd S)},
\end{multline}
where the $\cO$-coefficient depends only on $\|\partial^2_t\varphi\|_{\infty}$ and $\|K\|_{\infty}$. 
Furthermore, due to \eqref{E:DLtvarphi}, we have the estimate
$\|u\|^2_{L^2(\Sigma,\dd\Sigma)}\ge (1-C_2\delta)\|v\|^2_{L^2(S,\dd S)}$, \Gr where $C_{2}$ depends only on $\|\partial_{t}\varphi\|_\infty$.
This gives
\begin{multline*}
\dfrac{h^+_\alpha(u,u)}{\|u\|^2_{L^2(\Sigma,\dd\Sigma)}}-E^{D}\\
\begin{aligned}
\le&
\dfrac{(1+C_1\delta)\displaystyle\int_{S} g^{\rho \mu} \,\partial_\rho v \,\partial_\mu v\,\dd S
-\alpha \langle v, K v\rangle_{L^2(S,\dd S)}+\cO(1+\alpha e^{-\delta\alpha})\|v\|^2_{L^2(S,\dd S)}}
{(1-C_2\delta)\|v\|^2_{L^2(S,\dd S)}},
\end{aligned}
\end{multline*}
and we deduce the lemma by choosing $c_{1}^{+}=C_{1}$, and $c_{0}^{+}>0$ so that $(1-C_{2}\delta)^{-1}\leq 1+c_{0}^{+}\delta$, which is possible since $\delta$ becomes small as $\alpha$ tends to $+\infty$.
\end{proof}

\subsection{Proof of the upper bound}
\label{SS:proofub}
Now for each $j$ we can use the definition \eqref{E:QR} 
by testing on the subspaces $L\subset \cD(h^+_\alpha)$
of the form
\[
L=\{
u: u(s,t)=v(s)\psi(t) \text{ with } v\in \Lambda
\},
\]
where $\Lambda$ are the $j^{\mbox{th}}$ dimensional subspaces of $H^1(S)$, which is the form domain of $-\Delta_S-\alpha K$.
Lemma \ref{L:estimatetensor} then implies
\begin{equation}
\label{E:estimEmED}
E_j(H^+_\alpha)-E^{D}\le
(1+c_{0}^{+}\delta) E_j\big( -(1+c_{1}^{+}\delta)\Delta_S-\alpha K \big)+\cO(1+\alpha e^{-\delta\alpha}).
\end{equation}
The right-hand side can be estimated as follows:
\begin{lemma}
\label{L:Ejreduit}
For any fixed $j\in\NN$ there holds, as $\alpha\to+\infty$,
\[
(1+c_{0}^{+}\delta) E_j\big( -(1+c_{1}^{+}\delta)\Delta_S-\alpha K \big) \leq E_j(-\Delta_S-\alpha K)+\cO(\delta\alpha),
\]
where $\delta$ is a function of $\alpha$ satisfying \eqref{eq-dada}. The constants depend only on $\|L_{s}\|_{\infty}$, $\|K\|_{\infty}$, $\|\partial_{t}\varphi\|_{\infty}$, $\|\partial^2_t\varphi\|_{\infty}$.
\end{lemma}
\begin{proof}

We have
\begin{multline}
(1+c_{0}^{+}\delta) E_j\big( -(1+c_{1}^{+}\delta)\Delta_S-\alpha K \big)\\
\begin{aligned}
=&(1+c_{0}^{+}\delta) \Big(E_j\big( -(1+c_{1}^{+}\delta)\Delta_S+\alpha(K_\mx-K)\big)-\alpha K_\mx\Big)\\
=&(1+c_{0}^{+}\delta) E_j\big( -(1+c_{1}^{+}\delta)\Delta_S+\alpha(K_\mx-K)\big)-\alpha K_\mx+\cO(\delta\alpha)\\
\le& (1+c_{0}^{+}\delta) E_j\big( -(1+c_{1}^{+}\delta)\Delta_S+(1+c_{1}^{+}\delta)\alpha(K_\mx-K)\big)-\alpha K_\mx+\cO(\delta\alpha)\\
\le &(1+C\delta) E_j(-\Delta_S+\alpha(K_\mx-K)\big)-\alpha K_\mx+\cO(\delta\alpha)\\
= &E_j\big(-\Delta_S+\alpha(K_\mx-K)\big)-\alpha K_\mx\\
&\qquad +\cO\Big(\delta E_j(-\Delta_S+\alpha(K_\mx-K)\big)\Big)
+\cO(\delta\alpha)\\
=& E_j(-\Delta_S-\alpha K) +\cO\Big(\delta E_j\big(-\Delta_S+\alpha(K_\mx-K)\big)\Big)
+\cO(\delta\alpha).
\end{aligned}
\end{multline}
As $K$ is bounded, we have the rough estimate $E_j\big(-\Delta_S+\alpha (K_\mx-K)\big)=\cO(\alpha)$, 
and the remainder depends on the constants $c_{0}^{+}$, $c_{1}^{+}$ and $\|K\|_{\infty}$ only. 
\end{proof}
Finally, combining Lemma \ref{L:Ejreduit} with \eqref{E:estimEmED} and Lemma \ref{lemd}, we get
\begin{multline}
E_j(H^+_\alpha)\le E^{D}+E_j(-\Delta_S-\alpha K)+\cO(1+\delta\alpha+\alpha e^{-\delta\alpha})\\
=-\alpha^2 
+E_j(-\Delta_S-\alpha K)+\cO(1+\delta\alpha+\alpha^2 e^{-\delta\alpha}).
\end{multline}
In order to have an optimal remainder we take
\[
\delta=\dfrac{b \log\alpha}{\alpha}, \quad b\ge 2,
\]
then $E_j(H_\alpha)\le E_j(H^+_\alpha)\le -\alpha^2+E_j(-\Delta_S-\alpha K)+\cO(\log\alpha)$.

\section{Proof of Theorem \ref{thm1}: lower bound}

\subsection{Minoration of the quadratic form}
\label{SS:lb1}

The operator $T^\beta$ of Lemma~\ref{lemn} with $\beta=0$ will play a special role
and it will be denoted by $T^{N}$. The first eigenvalue and the first normalized eigenfunction
will be denoted \emph{in this section} by $E^{N}$ and $\psi$ respectively. Recall again that
$\delta$ and $\alpha$ obey \eqref{eq-dada}.

We represent any function $u\in \cD(h^-_\alpha)$ in the form
\begin{equation}
\label{E:decomposeu}
u(s,t)=v(s)\psi(t)+w(s,t)
\end{equation}
 with
\begin{equation}
      \label{eq-vv}
v(s):=\int_0^\delta \psi(t) u(s,t)\dd t, \quad v\in H^1(S).
\end{equation}
Remark that the both functions $(s,t)\mapsto v(s)\psi(t)$ and
$w$ are in $\cD(h^-_\alpha)$. The following proposition gives a lower
bound on the expression $h^-_\alpha(u,u)-E^{N}\|u\|^2_{L^2(\Sigma,\dd\Sigma)}$ in terms
of this decomposition.
\begin{prop}
\label{P:Lowerbound}
There exist positive constants $c_{0}^{-}$ and $c_{1}^{-}$ such that, as $\alpha\to+\infty$,
\begin{multline}
      \label{eq-hum3}
h^-_\alpha(u,u)-E^{N}\|u\|^2_{L^2(\Sigma,\dd\Sigma)}\\
\begin{aligned}
&\ge(1-c_{0}^{-}\delta)\int_{S} g^{\rho \mu} \,\partial_\rho v \,\partial_\mu v\,\dd S-\alpha \langle v, Kv\rangle_{L^2(S,\dd S)} -c_{1}^{-}(1+\alpha e^{-\delta\alpha})\|v\|^2_{L^2(S,\dd S)}\\
&\quad+\dfrac{\alpha^2}{2}\int_{S} \int_{0}^\delta  w(s,t)^2\dd t\, \dd S
\end{aligned}
\end{multline}
for any $u\in \cD(h^-_\alpha)$. The constants depend only on $\|L_{s}\|_{\infty}$, $\|\partial_{t}\varphi\|_{\infty}$ and $\|\partial_{t}^2\varphi\|_{\infty}$.
\end{prop}
The rest of this subsection is devoted to the proof of Proposition \ref{P:Lowerbound}. Using the decomposition \eqref{E:decomposeu}, we clearly have
\begin{align}
       \label{eq-hm1}
h^-_\alpha(u,u)&=(1-C_{g}\delta)\int_{\Sigma} g^{\rho \mu} \,\partial_\rho u \,\partial_\mu u\,\dd\Sigma 
+\int_{\Sigma} |\partial_t u|^2\dd\Sigma-\alpha \int_S u(s,0)^2 \dd S\\
 \label{E:sumI}
&=:I_1+I_2+I_3+I_4,
\end{align}
where we have set
\[
 \left\{
\begin{aligned}
&I_{1}=(1-C_{g}\delta)\int_0^\delta\int_{S} g^{\rho \mu} \,\partial_\rho u \,\partial_\mu u\,\varphi\,\dd S\dd t,
\\
&I_{2}=\int_{S} v(s)^2\int_0^\delta \psi'(t)^2\varphi(s,t)\dd t \,\dd S,
\\
&I_{3}=2 \bigg[
\int_{S} v(s)\int_0^\delta \psi'(s) \partial_t w(s,t)\varphi(s,t)\dd t\, \dd S
- \alpha \psi(0)\int_{S} v(s)w(s,0)\dd S
\bigg]
\\
&I_{4}= \int_{S} \int_0^\delta \big|\partial_t w(s,t)\big|^2\varphi(s,t)\dd t \dd S
-\alpha \psi(0)^2 \int_S v(s)^2\dd S -\alpha \int_S w(s,0)^2\dd S.
\end{aligned}
\right.
\]
We estimate the four terms separately.
\begin{lemma}
\label{L:I1}
There exists $C_{1}>0$ such that, as $\alpha\to+\infty$, 
\begin{equation}
       \label{term1}
I_1\ge
(1-C_{1}\delta)\int_{S} g^{\rho \mu} \,\partial_\rho v \,\partial_\mu v\,\dd S.
\end{equation}
\Gr Moreover, the constant $C_{1}$ depends only on $\|\partial_{t}\varphi\|_{\infty}$ and $\|L_{s}\|_{\infty}$
and is independent of~$u$.
\end{lemma}
\begin{proof}
Following the decomposition \eqref{E:decomposeu}, we get by using \eqref{E:DLtvarphi} a constant $C>0$ such that
\begin{equation}
     \label{eq-gg1}
\begin{split}     
\int_0^\delta\int_{S} & g^{\rho \mu}  \,\partial_\rho u(s,t) \,\partial_\mu u(s,t)\,\varphi(s,t)\dd S\dd t 
\\=&
\int_0^\delta\int_{S} \psi(t)^2 g^{\rho \mu} \,\partial_\rho v(s) \,\partial_\mu v(s)\,\varphi(s,t)\dd S\dd t+2 
\int_0^\delta\int_{S} \psi(t) g^{\rho \mu} \,\partial_\rho v(s) \,\partial_\mu w(s,t)\,\varphi\dd S\dd t
\\&+
\int_0^\delta\int_{S} g^{\rho \mu} \,\partial_\rho w(s,t) \,\partial_\mu w(s,t)\,\varphi\dd S\dd t\\
\ge&
(1-C\delta)\int_0^\delta\int_{S} \psi(t)^2 g^{\rho \mu} \,\partial_\rho v(s) \,\partial_\mu v(s)\dd S\dd t+2 
\int_0^\delta\int_{S} \psi(t) g^{\rho \mu} \,\partial_\rho v(s) \,\partial_\mu w(s,t)\,\varphi(s,t)\dd S\dd t\\
&+(1-C\delta)\int_0^\delta\int_{S} g^{\rho \mu} \,\partial_\rho w(s,t) \,\partial_\mu w(s,t)\dd S\dd t\\
=&
(1-C\delta)\int_{S} g^{\rho \mu} \,\partial_\rho v(s) \,\partial_\mu v(s)\dd S+2 
\int_0^\delta\int_{S} \psi(t) g^{\rho \mu} \,\partial_\rho v(s) \,\partial_\mu w(s,t)\,\varphi(s,t)\dd S\dd t\\
&+(1-C\delta)\int_0^\delta\int_{S} g^{\rho \mu} \,\partial_\rho w(s,t) \,\partial_\mu w(s,t)\,\dd S\dd t,
\end{split}
\end{equation}
\Gr where the constant $C$ depends only on $\|\partial_{t}\varphi\|_{\infty}$. \Bk Remark that for the function $w$ we have
\begin{equation}
     \label{eq-w}
     \int_0^\delta \psi(t)w(s,t)\dd t=0 \text{ and, hence,  }
\int_0^\delta \psi(t)\partial_\rho w(s,t)\dd t=0,
\quad s\in S.
\end{equation}
We deduce:
\begin{multline}
\int_0^\delta\int_{S} \psi g^{\rho \mu} \,\partial_\rho v \,\partial_\mu w\,\varphi\dd S\dd t\\
=
\int_0^\delta\int_{S} \psi g^{\rho \mu} \,\partial_\rho v \,\partial_\mu w\,\dd S\dd t
+
\int_0^\delta\int_{S} \psi g^{\rho \mu} \,\partial_\rho v \,\partial_\mu w\, (\varphi-1)\dd S\dd t\\
=\int_0^\delta\int_{S} \psi g^{\rho \mu} \,\partial_\rho v \,\partial_\mu w\, (\varphi-1)\dd S\dd t.
\end{multline}
Using again \eqref{E:DLtvarphi}, we estimate with the same constant $C$, using the Cauchy-Schwarz inequality
for the metric $(g^{\rho\mu})$,
\begin{multline}
\Big|\int_0^\delta\int_{S} \psi(t) g^{\rho \mu} \,\partial_\rho v(s) \,\partial_\mu w(s,t)\,(\varphi(s,t)-1)\dd S\dd t\Big|\\
\begin{aligned}
&\le
C\delta
\int_0^\delta\int_{S} 
\Big|g^{\rho \mu} \,\psi(t) \partial_\rho v(s) \,\partial_\mu w(s,t)\Big|\dd S\dd t\\
&\le
\dfrac{C\delta}{2}
\int_0^\delta\int_{S} \psi(t)^2
g^{\rho \mu} \,\partial_\rho v(s) \,\partial_\mu v(s)\dd S\dd t
+\dfrac{C\delta}{2}
\int_0^\delta\int_{S}g^{\rho \mu} \,\partial_\rho w(s,t) \,\partial_\mu w(s,t)\dd S\dd t\\
&=
\dfrac{C\delta}{2}
\int_{S} g^{\rho \mu} \,\partial_\rho v \,\partial_\mu v\,\dd S
+\dfrac{C\delta}{2}
\int_0^\delta\int_{S}g^{\rho \mu} \,\partial_\rho w(s,t) \,\partial_\mu w(s,t)\dd S\dd t
\end{aligned}
\end{multline}
which gives
\begin{multline}
2 \bigg[
\int_0^\delta\int_{S} \psi(t) g^{\rho \mu} \,\partial_\rho v(s) \,\partial_\mu w(s,t)\,\varphi(s,t)\dd S\dd t\bigg]\\
\le 
C\delta
\int_{S} g^{\rho \mu} \,\partial_\rho v(s) \,\partial_\mu v(s)\dd S
+C\delta
\int_0^\delta\int_{S}g^{\rho \mu} \,\partial_\rho w(s,t) \,\partial_\mu w(s,t)\dd S\dd t.
\end{multline}
Substituting the last inequality into \eqref{eq-gg1} we obtain, 
\begin{multline*}
\int_0^\delta\int_{S} g^{\rho \mu} \,\partial_\rho u \,\partial_\mu u\,\varphi\dd S\dd t 
\ge 
(1-2C\delta)\int_{S} g^{\rho \mu} \,\partial_\rho v \,\partial_\mu v\dd S
+(1-2C\delta)\int_0^\delta\int_{S} g^{\rho \mu} \,\partial_\rho w \,\partial_\mu w\,\dd S\dd t
\end{multline*}
and, therefore, for sufficiently small $\delta$:
\[
(1-C_{g}\delta)\int_0^\delta\int_{S} g^{\rho \mu} \,\partial_\rho u \,\partial_\mu u\,\varphi\dd S\dd t 
\ge(1-C_{g}\delta)(1-2C\delta)\int_{S} g^{\rho \mu} \,\partial_\rho v \,\partial_\mu v\dd S.
\]
The result follows as $C_{g}$ depends only on $\|L_{s}\|_{\infty}$, and $C$ depends only on $\|\partial_{t}\varphi\|_{\infty}$.
\end{proof}

\begin{lemma}
     \label{term2}
       There exists $C_{2}>0$ such that, as $\alpha\to+\infty$:
     \begin{multline}
I_2\ge  \alpha \psi(0)^2\int_S v(s)^2\dd S
+E^{N} \|u-w\|^2_{L^2(\Sigma,\dd\Sigma)}\\
-\alpha \langle v, Kv\rangle_{L^2(S,\dd S)}
-C_2(1+ \alpha e^{-\delta\alpha})\|v\|^2_{L^2(S,\dd S)}.
\end{multline}
The constant $C_{2}$ depends only on $\|K\|_{\infty}$, $\|\partial_{t}\varphi\|_{\infty}$ and $\|\partial_{t}^2\varphi\|_{\infty}$
and is independent of~$u$.
\end{lemma}
\begin{proof}
As in \eqref{eq-if1}, an integration by part leads to 
\begin{multline*}
\int_0^\delta \psi'(t)^2\varphi(s,t)\dd t
=
\alpha \psi(0)^2+ E^{N} \int_0^\delta \psi(t)^2 \varphi(s,t) \dd t
-\dfrac{K(s)}{2} \psi(0)^2\\
- \dfrac{\partial_t\varphi(s,\delta)}{2} \psi(\delta)^2
+\dfrac 12
\int_0^\delta \psi(t)^2 \partial^2_t\varphi(s,t)\dd t.
\end{multline*}
The additional term in comparison with \eqref{eq-if1} comes from the fact  that $\psi(\delta)\neq0$. We deduce: 
\begin{equation}
     \label{term2a}
\begin{split}
I_2&=\int_{S} v(s)^2\int_0^\delta \psi'(t)^2\varphi(s,t)\dd t \,\dd S\\
&=\alpha \psi(0)^2\int_S v^2\dd S +E^{N} \|u-w\|^2_{L^2(\Sigma,\dd\Sigma)} 
-\dfrac{\psi(0)^2}{2} \langle v, Kv\rangle_{L^2(S,\dd S)}\\
&\quad -
\dfrac{\psi(\delta)^2}{2} \langle v, \partial_t \varphi(\cdot,\delta) v\rangle_{L^2(S,\dd S)}
+
\dfrac 12
\int_S v(s)^2\int_0^\delta \psi(t)^2 \partial^2_t\varphi(s,t)\dd t \dd S.
\end{split}
\end{equation}
Due to \eqref{eq-fn1d}, $\psi(\delta)^2=O(\alpha e^{-2\delta\alpha})$, and there exists $C>0$ such that for $\alpha$ large enough
one has
\begin{equation}
\label{E:L13inter1}
\dfrac{\psi(\delta)^2}{2} \langle v, \partial_t \varphi(\cdot,\delta) v\rangle_{L^2(S,\dd S)}\le C \alpha e^{-2\delta\alpha}\|v\|^2_{L^2(S,\dd S)},
\end{equation}
where the constant $C$ depends only on $\|\partial_{t}\varphi\|_{\infty}$. We also have:
\begin{equation}
\label{E:L13inter2}
\Big|\int_0^\delta \psi(t)^2 \partial^2_t\varphi(s,t)\dd t\Big|
\le
\|\partial_{t}^2\varphi\|_{\infty}.
\end{equation}
Moreover, \eqref{eq-fn1} provides $C'>0$, depending only on $\|K\|_{\infty}$, such that for $\alpha$ large enough:
\begin{equation}
\label{E:majorpsi(0)}
\frac{\psi(0)^2}{2}\langle v, Kv\rangle_{L^2(S,\dd S)} \leq \alpha  \langle v, Kv\rangle_{L^2(S,\dd S)}+C'\alpha e^{-\delta\alpha}\|v\|^2_{L^2(S,\dd S)}.
\end{equation}
The lemma follows by combining \eqref{E:L13inter1}--\eqref{E:majorpsi(0)} with \eqref{term2a}.
\end{proof}
The crossed term $I_{3}$ needs a parametric estimate:
\begin{lemma}
\label{L:I3}
There exists  $C_{3}>0$ such that for any $r>0$ for large $\alpha$ one has
\[
I_{3} \geq
 2E^{N} \langle u-w,w\rangle_{L^2(\Sigma,\dd\Sigma)}-C_{3}r\alpha^2 \|v\|^2_{L^2(S,\dd S)}
-C_{3}\dfrac{1}{r}\int_S \int_{0}^{\delta}w(s,t)^2 \dd t \dd S.
\]
Moreover, the constant $C_{3}$ depends only on $\|\partial_{t}\varphi\|_{\infty}$ and does not depend on $u$. 
\end{lemma}
\begin{proof}
Using the integration by parts we have:
\begin{equation}
        \label{eq-i3aa}
\begin{aligned}
I_3&=2\int_{S} v(s)\int_0^\delta \psi'(s) \partial_t w(s,t)\varphi(s,t)\dd t\, \dd S
- 2 \alpha \psi(0)\int_{S} v(s)w(s,0)\dd S\\
&=2\int_S v(s)
\bigg(
\Big[\psi'(t)w(s,t)\varphi(s,t)\Big]_{t=0}^{t=\delta} -\int_{0}^{\delta} \psi''(t) w(s,t)\varphi(s,t)\dd t\\
&\quad
-\int_{0}^{\delta} \psi'(t) w(s,t)\partial_t \varphi(s,t)\dd t
\bigg)\dd S- 2 \alpha \psi(0)\int_{S} v(s)w(s,0)\dd S\\
&=2\alpha \psi(0)\int_S v(s)w(s,0)\varphi(s,0)\dd S+2E^{N} \int_S\int_0^\delta v(s)\psi(t) w(s,t)\varphi(s,t)\dd t\dd S\\
&\quad -2\int_S\int_{0}^{\delta}v(s) \psi'(t) w(s,t)\partial_t \varphi(s,t)\dd t\dd S- 2 \alpha \psi(0)\int_{S} v(s)w(s,t)\dd S\\
&= 2E^{N} \langle u-w,w\rangle_{L^2(\Sigma,\dd\Sigma)}-2\int_S\int_{0}^{\delta} v(s) \psi'(t) w(s,t)\partial_t \varphi(s,t)\dd t\dd S,
\end{aligned}
\end{equation}
where we have used the boundary conditions $\psi'(\delta)=0$ and $\psi'(0)=-\alpha\psi(0)$. 

We estimate now, with any $r>0$:
\begin{multline}
       \label{eq-i32}
\Big|2\int_S\int_{0}^{\delta} v(s)\psi'(t) w(s,t)\partial_t \varphi(s,t)\dd t\dd S\Big|\\
\begin{aligned}
&\le
\|\partial_{t}\varphi\|_{\infty}\int_S\int_{0}^{\delta} 2\Big|v(s)\psi'(t) w(s,t) \Big|\dd t\dd S\\
&\le \|\partial_{t}\varphi\|_{\infty}r
\int_S\int_{0}^{\delta} \psi'(t)^2 v(s)^2 \dd t\dd S
+\dfrac{\|\partial_{t}\varphi\|_{\infty}}{r} \int_S\int_{0}^{\delta}  w(s,t)^2 \dd t \dd S\\
&\le 2\|\partial_{t}\varphi\|_{\infty}r\alpha^2 \|v\|^2_{L^2(S,\dd S)}
+\dfrac{\|\partial_{t}\varphi\|_{\infty}}{r} \int_S\int_{0}^{\delta}  w(s,t)^2 \dd t \dd S,
\end{aligned}
\end{multline}
where we have used $\|\psi'\|^2_{L^2(0,\delta)} \leq 2 \alpha^2$ for $\alpha$ large enough, see \eqref{eq-fn2}. The substitution of \eqref{eq-i32} into \eqref{eq-i3aa}
gives the lemma by choosing $C_{3}=2\|\partial_{t}\varphi\|_{\infty}$.
\end{proof}

We are now able to finish the proof of Proposition \ref{P:Lowerbound}. We use Lemmas \ref{L:I1}--\ref{L:I3} in \eqref{E:sumI} and deduce
\begin{equation}
      \label{eq-hum1}
      \begin{aligned}
h^-_\alpha(u,u)&\ge
(1-C_{1}\delta)\int_{S} g^{\rho \mu} \,\partial_\rho v \,\partial_\mu v\,\dd S-\alpha\langle v, Kv\rangle_{L^2(S,\dd S)}\\
&+E^{N} \|u-w\|^2_{L^2(\Sigma,\dd\Sigma)} +2E^{N} \langle u-w,w\rangle_{L^2(\Sigma,\dd\Sigma)}+\int_{S} \int_0^\delta \big|\partial_t w(s,t)\big|^2\varphi(s,t)\dd t \dd S\\
&-\alpha \int_S |w(s,0)|^2\dd S -C_{4}(1+ r\alpha^2+\alpha e^{-\delta\alpha})\|v\|^2_{L^2(S,\dd S)} -\dfrac{C_{3}}{r} \int_{S} \int_{0}^\delta \big| w(s,t)\big|^2\dd t\, \dd S,
\end{aligned}
\end{equation}
where $C_{4}=\max(C_{2},C_{3})$. We have the equality
\begin{equation}
\label{eq-upr}
\|u-w\|^2_{L^2(\Sigma,\dd\Sigma)}+2 \langle u-w,w \rangle_{L^2(\Sigma,\dd\Sigma)}=
\|u\|^2_{L^2(\Sigma,\dd\Sigma)} -\|w\|^2_{L^2(\Sigma,\dd\Sigma)}.
\end{equation}
Due to \eqref{eq-w}, we can use \eqref{eq-fpsi}, so that
$$\int_{S} \int_{0}^\delta \big|\partial_t w(s,t)\big|^2\dd t\, \dd S
-\alpha \int_S w(s,0)^2\dd S\geq0,$$ 
and, therefore, substituting \eqref{eq-upr} into \eqref{eq-hum1}, we get
\begin{multline}
      \label{eq-hum2}
h^-_\alpha(u,u)-E^{N}\|u\|^2_{L^2(\Sigma,\dd\Sigma)}
\ge(1-C_{1}\delta)\int_{S} g^{\rho \mu} \,\partial_\rho v \,\partial_\mu v\,\dd S-\alpha \langle v, Kv\rangle_{L^2(S,\dd S)}
\\ \quad-C_{4}(1+r\alpha^2+ \alpha e^{-\delta\alpha})\|v\|^2_{L^2(S,\dd S)}
-E^{N}\|w\|^2_{L^2(\Sigma,\dd\Sigma)}-\dfrac{C_{3}}{r}\int_{S} \int_{0}^\delta  w(s,t)^2\dd t\, \dd S.
\end{multline}
Due to \eqref{E:DLtvarphi}, we have 
\[
\|w\|^2_{L^2(\Sigma,\dd\Sigma)} \geq (1-\|\partial_{t}\varphi\|_{\infty}\delta) \int_{S} \int_{0}^\delta  w(s,t)^2\dd t\, \dd S.
\]
We choose $r=3C_{3}/\alpha^2$ in \eqref{eq-hum2}, so that the asymptotic expansion \eqref{E:asymptoEbeta}
for $E^{N}$  provides a constant $C_{5}>0$ such that for $\alpha$ large enough: 
\begin{multline*}
h^-_\alpha(u,u)-E^{N}\|u\|^2_{L^2(\Sigma,\dd\Sigma)}
\ge(1-C_{1}\delta)\int_{S} g^{\rho \mu} \,\partial_\rho v \,\partial_\mu v\,\dd S-\dfrac{\psi(0)^2}{2} \langle v, Kv\rangle_{L^2(S,\dd S)}
\\ \quad-C_{5}(1+ \alpha e^{-\delta\alpha})\|v\|^2_{L^2(S,\dd S)}
+\frac{\alpha^2}{2}\int_{S} \int_{0}^\delta w(s,t)^2\dd t\, \dd S.
\end{multline*}
Therefore, the proof is concluded by setting $c_{0}^{-}=C_{1}$ and $c_{1}^{-}=C_{5}$. Noticing that $C_{4}$ and $C_{5}$ express with $C_{2}$ and $C_{3}$,
we deduce that the constants depends only on $\|L_{s}\|_{\infty}$, $\|\partial_{t}\varphi\|_{\infty}$ and $\|\partial_{t}^2\varphi\|_{\infty}$. 

\subsection{Asymptotics of $E_j$}
\label{SS:lb2}
The expression on the right-hand side  of \eqref{eq-hum3} can be viewed as a lower semibounded quadratic form defined on
$\cD(h^-_{\alpha})\subset L^2(\Sigma,\dd S \dd t)$.
Denote its closure in $L^2(\Sigma,\dd S \dd t)$ by $q$, and let $Q$ be the associated self-adjoint operator in
$L^2(\Sigma,\dd S \dd t)\equiv L^2(S,\dd S)\otimes L^2(0,\delta)$.
It writes as
\[
Q=\Big[-(1-c_{0}^{-}\delta)\Delta_S -\alpha K -c_{1}^{-}(1+\alpha e^{-2\delta\alpha}) \Big] P + \dfrac{\alpha^2}{2}\, (1-P),
\]
where $P:L^2(\Sigma,\dd S \dd t)\to L^2(S,\dd S)\otimes \psi$ is the orthogonal projector $(Pu)(s,t):=v(s)\psi(t)$ with $v$ defined in \eqref{eq-vv}.
For each fixed $j$ and large $\alpha$ we have 
\[
E_j\Big(-(1-c_{0}^{-}\delta)\Delta_S -\alpha K -c_{1}^{-}(1+\alpha e^{-2\delta\alpha})  \Big)=\cO(\alpha) < \dfrac{\alpha^2}{2},
\]
hence,
\[
E_j(Q)=E_j\Big(-(1-c_{0}^{-}\delta)\Delta_S -\alpha K -c_{1}^{-}(1+\alpha e^{-\delta\alpha}) \Big).
\]
Furthermore, using
\[
\|u\|^2_{L^2(\Sigma,d\Sigma)}\le (1+ \|\partial_{t}\varphi\|_{\infty} \delta) \|u\|^2_{L^2(\Sigma,\dd S\dd t)}
\]
we have a positive constant $C_{\varphi}$, depending only on $\|\partial_{t}\varphi\|_{\infty}$ such that
\[
\dfrac{h^-_\alpha(u,u)}{\|u\|^2_{L^2(\Sigma,d\Sigma)}}-E^{N}\\
\ge (1-C_{\varphi} \delta) \dfrac{q(u,u)}{\|u\|^2_{L^2(\Sigma,\dd S\dd t)}}.
\]
As the identification operator $f\mapsto f$ defines an injection 
of $\cD(h^-_\alpha)\subset L^2(\Sigma,d\Sigma)$ in $\cD(q)\subset L^2(\Sigma,\dd S \dd t)$,
it follows  that
\begin{multline}
        \label{eq-ee}
E_j(H^-_\alpha)\ge (1-C_{\varphi} \delta)E_j(Q)+E^{N}\\
=-\alpha^2+(1-C_{\varphi} \delta)E_j\Big(-(1-c_{0}^{-}\delta)\Delta_S -\alpha K\Big) +\cO(1+\alpha^2 e^{-\delta\alpha}),
\end{multline}
where we have used the asymptotics  \eqref{E:asymptoEbeta} for $E^N$. In addition, by Lemma \ref{L:Ejreduit}, we get
\begin{equation*}
(1-C_{\varphi} \delta) E_j\Big(-(1-c_{0}^{-}\delta)\Delta_S -\alpha K\Big)= E_j(-\Delta_S - \alpha K) +\cO(\delta\alpha).
\end{equation*}
Hence, by substituting in \eqref{eq-ee},
\[
 E_j(H^-_\alpha)\geq-\alpha^2 +E_j(-\Delta_S - \alpha K)
+\cO(1+\alpha^2 e^{-\delta\alpha}+\delta\alpha),
\]
and the constants depend only on $\|L_{s}\|_{\infty}$, $\|K\|_{\infty}$, $\|\partial_{t}\varphi\|_{\infty}$ and $\|\partial_{t}^2\varphi\|_{\infty}$.
Choosing
\begin{equation}
   \label{eq-dba}
\delta=\dfrac{b \log \alpha}{\alpha}, \quad b\ge 2,
\end{equation}
we arrive at the result.

\section{Proof of Theorem~\ref{thm2}}\label{th2pr}
The main idea for improving the remainder estimate
is to work in unweighted spaces from the very beginning.
The weight $\varphi$ is indeed $C^{1}$ with respect to the $s$ variable now,
and this allows for more precise Taylor expansions of $\varphi$ in $\Sigma$, so that the comparison
between the Robin Laplacian and the decoupled operator becomes more precise.
We start with the following simple result:
\begin{lemma}\label{lem23}
Under the assumption of Theorem~\ref{thm2}, for any fixed $j\in\NN$
one has
\[
E_j(-\Delta_S -\alpha K)=-\alpha K_{\max} + \cO(\alpha^{2/3})
\text{ as } \alpha\to+\infty.
\]
\end{lemma}

\begin{proof}
Due to $(-\Delta_S)\ge 0$ we have the obvious lower bound $E_j (-\Delta_S -\alpha K)\ge -\alpha K_\mathrm{max}$.
Let us prove the upper bound. For $s,s_0\in S$, let $d(s,s_0)$ denote the geodesic distance between $s$ and $s_0$.
Let $s_0\in S$ be such that $K(s_0)=K_\mathrm{max}$. As $K$ is at least $C^1$, there exist
$\varepsilon>0$ and $C>0$ such that
\begin{equation}
 \label{eq-ktay}
K(s)\ge K_\mathrm{max} - C d(s,s_0) \text{ as }d(s,s_0)<\varepsilon. 
\end{equation}
Now let us choose $j$ functions $f_1,\dots, f_j \in C^\infty_c(\RR_+)$
having disjoint supports, non identically zero,
and set  $v_i(s)= f_i\big(r^{-1}d(s,s_0)\big)$, where $r>0$ is small and will be chosen later.
For small $r$, the functions $v_i$ have pairwise disjoint supports and belong to the domain of $\Delta_S$.
In particular, they are linearly independent,
and 
\[
\big\langle
v_i, (-\Delta_S-\alpha K)v_l\big\rangle=0 \text{ for } i\ne l.
\]
On the other hand, 
\[
\theta_i(r):=\int_S v_i(s)^2 \dd S = \int_S f_i\big( r^{-1}d(s,s_0)\big)^2 \dd S =
a_i r^{\nu-1} + o(r^{\nu-1}), \quad a_i>0.
\]
Using \eqref{eq-ktay} we have
\begin{multline*}
\big\langle v_i, (-\Delta_S-\alpha K)v_{i}\big\rangle
=\int_S g^{\rho\mu} \partial_\rho v_i \partial_\mu v_i \dd S -\alpha
\int_S K v_i^2 \dd S \\
\le b_i r^{\nu-3} -\alpha K_\mathrm{max} \theta_i(r) + c_i\alpha r \theta_i(r),
\quad b_i, \, c_i> 0,
\end{multline*}
which gives
\[
\dfrac{\big\langle v_i, (-\Delta_S-\alpha K)v_{i}\big\rangle}{\|v_i\|^2}
\le -\alpha K_\mathrm{max} + A_i (r^{-2} + \alpha r), \quad A_i>0.
\]
Now it is sufficient to take $r:=\alpha^{-1/3}$ and to test in
 \eqref{E:QR} on the subspace $L$ spanned by $v_1,\dots,v_j$.
\end{proof}

\subsection{Toward unweighted spaces}
In order to remove the weight $\varphi$, we perform the unitary transform 
\[
\Theta: L^2(\Sigma,\dd S\dd t)\ni u\mapsto \varphi^{-1/2}u \in L^2(\Sigma, \dd \Sigma),
\]
and consider the quadratic forms $u\mapsto h_{\alpha}^{\pm}(\Theta u,\Theta u)$ defined on $\Theta^{-1}\big(\cD(h_{\alpha}^{\pm})\big)\subset L^2(\Sigma,\dd S\dd t)$.
In order to reduce the analysis to decoupled operators, we prove approximation lemmas:
\begin{lemma}
There exists $\delta_{0}>0$ and positive  constants $C$ and $C'$ such that for all $\delta\in (0,\delta_{0})$ and $u\in \Theta^{-1}\big(\cD(h_{\alpha}^{\pm})\big)$
one has
\begin{multline}
\left|\int_{\Sigma}g^{\rho\mu}\partial_{\rho}(\varphi^{-1/2}u) \partial_{\mu}(\varphi^{-1/2}u)\varphi\dd S \dd t-\int_{\Sigma} g^{\rho\mu}\partial_{\rho}u\partial_{\mu}u\dd S \dd t \right|\\
\leq C\delta \int_{\Sigma} g^{\rho\mu}\partial_{\rho}u\partial_{\mu}u\dd S \dd t+C' \delta \| u \|_{L^2(\Sigma,\dd S \dd t)}^2,
 \end{multline}
and the constants depend only on $\|\varphi\|_{\infty}$ and $\|\nabla_{s}\partial_{t}(\varphi^{-1/2})\|_{\infty}$.  
\end{lemma}
\begin{proof}
We compute 
\begin{multline}
\label{E:developpegrhomu}
\int_{\Sigma}g^{\rho\mu}\partial_{\rho}(\varphi^{-1/2}u) \partial_{\mu}(\varphi^{-1/2}u)\varphi\dd S \dd t-\int_{\Sigma} g^{\rho\mu}\partial_{\rho}u \partial_{\mu}u \dd S \dd t\\
=\int_{\Sigma}g^{\rho\mu}\partial_{\rho}\varphi^{-1/2}\partial_{\mu}\varphi^{-1/2}\varphi u^2 \dd S \dd t+2 \int g^{\rho\mu}\partial_{\rho}\varphi^{-1/2}u\varphi^{1/2}\partial_{\mu}u \dd S \dd t.
\end{multline}
Due to \eqref{eq-r1}, we have the expansion
\[
\varphi(s,t)^{-1/2}=1+tA(s,t),
\]
where $A$ and its gradient are bounded. In particular, there exists $C_{0}>0$ with
\begin{equation}
\label{E:Mphis}
 \|\nabla_{s}\varphi^{-1/2}\|_{\infty} \leq C_{0} \delta,
\end{equation}
where $C_{0}$ is controlled by $\|\partial_{t}\nabla_{s}\varphi^{-1/2}\|_\infty$, see \eqref{E:UBps}.
We deduce that 
\begin{equation}
\label{E:MT1delta2}
\left|\int_{\Sigma}g^{\rho\mu}\partial_{\rho}\varphi^{-1/2}\partial_{\mu}\varphi^{-1/2}\varphi u^2 \dd S \dd t\right|
\leq
 C_{1}\delta^2 \|u\|^2_{L^2(\Sigma,\dd S\dd t)},
 \end{equation}
where the constant $C_{1}$ is controlled by $\|\nabla_{s}\partial_{t}\varphi^{-1/2}\|_{\infty}$ and $\|\varphi\|_{\infty}$. 
Using the Cauchy-Schwarz inequality for the metric $(g^{\rho\mu})$, we get
\begin{multline*}
\left|2\int_{\Sigma} g^{\rho\mu}\left(\partial_{\rho}\varphi^{-1/2}u\right)\left(\varphi^{1/2}\partial_{\mu}u\right) \dd S \dd t\right|\\
\begin{aligned}
& \leq 2\left( \int_{\Sigma}g^{\rho\mu}\partial_{\rho}\varphi^{-1/2}\partial_{\mu}\varphi^{-1/2}u^2  \dd S \dd t \int_{\Sigma}g^{\rho\mu}\partial_{\rho}u\partial_{\mu}u \varphi \dd S \dd t\right)^{1/2}
\\
& \leq \delta^{-1}\int_{\Sigma}g^{\rho\mu}\partial_{\rho}\varphi^{-1/2}\partial_{\mu}\varphi^{-1/2}u^2  \dd S \dd t+\delta\int_{\Sigma}g^{\rho\mu}\partial_{\rho}u\partial_{\mu}u \varphi \dd S \dd t\\
& \leq \delta^{-1} \|\nabla_s \varphi^{-1/2}\|_\infty^2
\int_{\Sigma}u^2  \dd S \dd t+\delta\int_{\Sigma}g^{\rho\mu}\partial_{\rho}u\partial_{\mu}u \varphi \dd S \dd t
\\
& \leq C_{1}\delta \|u\|^2_{L^2(\Sigma,\dd S\dd t)}+C_{2}\delta \int_{\Sigma}g^{\rho\mu}\partial_{\rho}u\partial_{\mu}u \dd S \dd t;
\end{aligned}
\end{multline*}
on the last step we have used \eqref{E:Mphis}.
We deduce the lemma by combining the last inequality with \eqref{E:developpegrhomu} and \eqref{E:MT1delta2}. \Gr Since $C_{2}$ is controlled by $\|\varphi\|_{\infty}$, we deduce that the constants depends only on $\|\varphi\|_{\infty}$ and $\|\nabla_{s}\partial_{t}(\varphi^{-1/2})\|_{\infty}$. \Bk
\end{proof}
\begin{lemma}
There exists $\delta_{0}>0$ and positive constants $C$ and $\beta$ such that for all $\delta\in (0,\delta_{0})$
there holds
\begin{multline}
\label{E:TUS1}
\int_{\Sigma}|\partial_{t}(\varphi^{-1/2}u)|^2 \varphi\dd S \dd t
\geq 
\int_{\Sigma}|\partial_{t}u|^2 \dd S \dd t-\int_{S}\frac{K(s)}{2}\,u(s,0)^2 \dd S\\
-\beta \int_{S}u(s,\delta)^2 \dd S-C\|u\|^2_{L^2(\Sigma,\dd S \dd t)}
\text{ for all } u\in \Theta^{-1}\big(\cD(h_{\alpha}^{-})\big)
\end{multline}
and
\begin{multline}
\label{E:TUS2}
\int_{\Sigma}\big|\partial_{t}(\varphi^{-1/2}u)\big|^2 \varphi\dd S \dd t
\leq 
\int_{\Sigma}|\partial_{t}u|^2 \dd S \dd t -\int_{S}\frac{K(s)}{2}u(s,0)^2 \dd S\\+C\|u\|^2_{L^2(\Sigma,\dd S \dd t)}
\text{ for all }  u\in \Theta^{-1}\big(\cD(h_{\alpha}^{+})\big),
\end{multline}
and the constants depend on $\|\partial_{t}(\varphi^{-1/2})\|_{\infty}$, $\|\varphi^{1/2}\|_{\infty}$ and $\|\partial_{t}(\varphi^{1/2}\partial_{t}\varphi^{-1/2})\|_{\infty}$ only.
\end{lemma}
\begin{proof}
We have 
\begin{multline}
\label{E:developparttUv}
\int_{\Sigma} |\partial_{t}(\varphi^{-1/2}u)|^2 \varphi\dd S \dd t  -\int_{\Sigma} |\partial_{t}u|^2 \dd S \dd t\\
=
\int_{\Sigma} |\partial_{t}(\varphi^{-1/2})|^2 u^2 \varphi\dd S \dd t+2\int_{\Sigma}\partial_{t}(\varphi^{-1/2})\varphi^{1/2}u\partial_{t}u \dd S \dd t,
\end{multline}
and there exists $C_{0}>0$ such that
\begin{equation}
\label{E:MajT1L3}
\left| \int_{\Sigma} \big|\partial_{t}(\varphi^{-1/2})\big|^2 u^2 \varphi\dd S \dd t\right|
\leq
C_{0}\|u\|^2_{L^2(\Sigma,\dd S \dd t)}.
\end{equation}
The second term is treated by an integration by parts: 
\begin{multline}
\label{E:ipptermedt} 
2\int_{\Sigma}\partial_{t}(\varphi^{-1/2})\varphi^{1/2}u\partial_{t}u \dd S \dd t
=
\int_{S}\int_{0}^{\delta}\partial_{t}(\varphi^{-1/2}) \varphi^{1/2}\partial_{t}(u^2)\dd t \dd S
\\ =
\int_{S}\Big[ \partial_{t}(\varphi^{-1/2})\varphi^{1/2}u^2 \Big]_{t=0}^{t=\delta}\dd S-\int_{\Sigma}\partial_{t}(\partial_{t}(\varphi^{-1/2})\varphi^{1/2})u^2 \dd S \dd t.
\end{multline}
Due to \eqref{eq-r1}, we have the expansion 
\[
\partial_{t}(\varphi^{-1/2})(s,t)
=
\frac{K(s)}{2}+tQ(s,t),
\]
where $Q$ is bounded in $\Sigma$, so that $\partial_{t}(\varphi^{-1/2})(s,0)=K(s)/2$, and \eqref{E:ipptermedt} provides
\[
\left|2\int_{\Sigma}\partial_{t}(\varphi^{-1/2})\varphi^{1/2}u\partial_{t}u \dd S \dd t + \int_{S}\dfrac{K(s)}{2}u(s,0)^2 \dd S \right|
 \leq 
C_{1}\|u\|^2_{L^2(\Sigma,\dd s \dd t)}+\beta \int_{S}u(s,\delta)^2 \dd S,
\]
where
\[
\beta=\sup_{s\in S}\Big|(\varphi^{1/2}\partial_{t}\varphi^{-1/2})(s,\delta)\Big|.
\]
By combining this with \eqref{E:developparttUv} and \eqref{E:MajT1L3}, we deduce the lower bound \eqref{E:TUS1}, and also the upper bound \eqref{E:TUS2} since $u(s,\delta)=0$ for $u\in \Theta^{-1}\big(\cD(h_{\alpha}^{+})\big)$.
 Moreover, the constant $C_{0}$ is controlled by $\|\partial_{t}\varphi^{-1/2}\|_{\infty}^2$, the constant $\beta$ by $\|\partial_{t}\varphi^{-1/2}\varphi^{1/2}\|_{\infty}$ and the constant $C_{1}$ by $\|\partial_{t}(\partial_{t}\varphi^{-1/2}\varphi^{1/2})\|_{\infty}$.
\end{proof}

We deduce by combining the last two lemmas that there exist positive constants $c_{0}$ and $c_{1}$ such that for all $\delta \in (0,\delta_{0})$:
\begin{equation}
\label{D:Qam}
\begin{aligned}
h_{\alpha}^{-}(\Theta u, \Theta u) 
&\geq
(1-c_{0}\delta)\int_{\Sigma}g^{\rho\mu}\partial_{\rho}u\partial_{\mu}u \dd S \dd t+\int_{\Sigma}|\partial_{t}u|^2 \dd S \dd t
\\
&\quad-\int_{S}\left(\alpha+\frac{K}{2}\right)u(s,0)^2 \dd S-\beta \int_{S}u(s,\delta)^2 \dd S-c_{1}\|u\|^2_{L^2(\Sigma,\dd S\dd t}\\
&\quad \text{ for } u\in \Theta^{-1}(\cD(h_{\alpha}^{-})\big),
\end{aligned}
\end{equation}
and
\begin{equation}
\label{D:Qap}
\begin{aligned}
h_{\alpha}^{+}(\Theta u, \Theta u) 
&\leq
(1+c_{0}\delta)\int_{\Sigma}g^{\rho\mu}\partial_{\rho}u\partial_{\mu}u \dd S \dd t+\int_{\Sigma}|\partial_{t}u|^2 \dd S \dd t\\
&\quad-\int_{S}\left(\alpha+\frac{K}{2}\right)u(s,0)^2 \dd S+c_{1}\|u\|^2_{L^2(\Sigma,\dd S \dd t)} \text{ for } u\in \Theta^{-1}\big(\cD(h_{\alpha}^{+})\big).
\end{aligned}
\end{equation}
We denote by $q_{\alpha}^{\pm}$ the quadratic forms on the right-hand side of \eqref{D:Qam} and \eqref{D:Qap} respectively, defined on the form domains $\cD(q_{\alpha}^{\pm}):=\Theta^{-1}(\cD(h_{\alpha}^{\pm}))$. The associated self-adjoint operators, both acting on the unweighted space $L^2(\Sigma, \dd S \dd t)$, will be denoted by $Q_{\alpha}^{\pm}$. Due to \eqref{eq-est2} one has
\begin{equation}
       \label{eq-est3}
E_j(Q_{\alpha}^{-})\le E_j(Q^\Omega_\alpha)\le E_j(Q_{\alpha}^{+})
\text{ for all $j$ with } E_j(Q_{\alpha}^{+})<0.
\end{equation}

\subsection{Upper bound}
Once again we estimate the quadratic form $q_{\alpha}^{+}$ evaluated on the functions $u$ that write as a product
$u(s,t)=v(s)\psi(t)$, where $\psi$ is a normalized eigenfunction of $T^{D}$ associated with $E^{D}$ (see Lemma~\ref{lemd})
and $v\in H^1(S)$. Here we have simply $\|u\|_{L^2(\Sigma,\dd S\dd t)}=\|v\|^2_{L^2(S,\dd S)}$, and
\[
q_{\alpha}^{+}(u,u)-E^{D}\|u\|^2_{L^2(\Sigma,\dd S\dd t)} =(1+c_{0}\delta)\int_{S}g^{\rho\mu}\partial_{\rho}v\partial_{\mu}v \dd S-\frac{\psi(0)^2}{2}\int_{S}K v^2 \dd S+c_{1}\|v\|_{L^2(S,\dd S)}^2.
\]
Using  \eqref{E:estimeED} we obtain 
\begin{equation}
\label{E:estimEmEDbis}
E_j(Q^+_\alpha)
\leq -\alpha^2+
 E_j\big( -(1+c_{0}\delta)\Delta_S-\alpha K \big)+\cO(1+\alpha e^{-\delta\alpha}).
\end{equation}
To estimate the right-hand side of \eqref{E:estimEmEDbis} we need an additional assertion:
\begin{lemma}
\label{L:vpopC3}
For any $j\in \NN$ there exist positive constants $C$, $\alpha_{0}$ and $\delta_0$
such that for  $\delta\in(0,\delta_0)$ and $\alpha\geq\alpha_{0}$ the following inequalities hold:
\begin{align}
\label{E:ubvpopC3}
E_j\big( -(1+c_{0}\delta)\Delta_S-\alpha K \big) &\leq E_j\big( -\Delta_S-\alpha K \big)+C\delta \alpha^{2/3},\\
\label{E:lbvpopC3}
E_j\big( -(1-c_{0}\delta)\Delta_S-\alpha K \big) &\geq E_j\big( -\Delta_S-\alpha K \big)-C\delta \alpha^{2/3}.
\end{align}
\end{lemma}
\begin{proof}
We only prove the upper bound, the lower bound being symmetric. We have 
\begin{align*}
E_j\big( -(1+c_{0}\delta)\Delta_S-\alpha K \big)
&=
(1+c_{0}\delta)E_{j}\big( -\Delta_S+\frac{1}{1+c_{0}\delta}\alpha (K_{\max}-K) \big)-\alpha K_{\max}
\\ & \leq
(1+c_{0}\delta)E_{j}\big( -\Delta_S+\alpha (K_{\max}-K) \big)-\alpha K_{\max}
\\ &=
E_{j}\big( -\Delta_S-\alpha K \big)+c_{0}\delta E_{j}\big(-\Delta_S+\alpha (K_{\max}-K)\big),
\end{align*}
and it is sufficient to apply Lemma~\ref{lem23} to the last term.
\end{proof}
Now let us assume that $\delta$ is a function of $\alpha$ satisfying \eqref{eq-dada}. 
Applying Lemma~\ref{L:vpopC3} to \eqref{E:estimEmEDbis} we deduce, for $\alpha\to+\infty$,
\begin{equation}
\label{E:estimEimproved}
E_j(Q^+_\alpha)
\le
 -\alpha^2+ 
 E_j\big( -\Delta_S-\alpha K \big)+\cO(\delta\alpha^{2/3}+1+\alpha^2 e^{-\delta\alpha}).
\end{equation}
Choosing $\delta=\alpha^{-\kappa}$ with $\kappa\in [2/3,1)$ and using \eqref{eq-est3} we obtain the result.

\subsection{Lower bound}
Similarly to Section \ref{SS:lb1}, we decompose any function $u\in \cD(q_{\alpha}^{-})$ as 
\[
u(s,t)=v(s)\psi(t)+w(s,t),
\]
where $\psi=\psi^{\beta}$ is a normalized eigenfunction of the operator $T^{\beta}$ associated with the first eigenvalue $E^{\beta}$, see Lemma \ref{lemn}, and 
\[
v(s)=\int_{0}^{\delta}\psi(t)u(s,t) \dd t.
\]
It follows that
 \begin{equation}
 \label{E:wperppsi}
\int_{0}^{\delta}\psi(t)w(s,t) \dd t=0, \quad s\in S, 
 \end{equation}
 which provides
  \begin{equation}
 \label{E:crossT01}
 \int_{\Sigma}v(s) \psi(t) w(s,t)\dd S \dd t=0 \text{ and } \int_{\Sigma}g^{\rho\mu}\partial_{\rho}v(s)\psi(t)\partial_{\mu}w(s,t) \dd S\dd t=0.
 \end{equation}
A direct computation provides 
 \begin{equation}
 \label{E:DevLBI}
\begin{aligned}
 q_{\alpha}^{-}(u,u)=&\,(1-c_{0}\delta)\int_{S}g^{\rho\mu}\partial_{\rho}v\partial_{\mu}v \dd S +(1-c_{0}\delta)\int_{\Sigma}g^{\rho\mu}\partial_{\rho}w(s,t)\partial_{\mu}w(s,t) \dd S \dd t
 \\
 &+\int_{\Sigma}v(s)^{2}\psi'(t)^2 \dd S \dd t+2 \int_{\Sigma}v(s)\psi'(t)\partial_{t}w(s,t) \dd S \dd t
\\
&+\int_{\Sigma}|\partial_{t}w(s,t)|^2 \dd S \dd t  -\int_{S}\Big(\alpha+\dfrac{K(s)}{2}\Big) v(s)^2\psi(0)^2 \dd S\\
&-2 \int_{S}\Big(\alpha+\dfrac{K(s)}{2}\Big) v(s)\psi(0)w(s,0) \dd S - \int_{S}\Big(\alpha+\dfrac{K(s)}{2}\Big)w(s,0)^2 \dd S
 \\
 &-\beta\int_{S}v(s)^2\psi(\delta)^2 \dd S-2\beta\int_{S} v(s)\psi(\delta)w(s,\delta) \dd S-\beta \int_{S}w(s,\delta)^2 \dd S\\
&-c_{1}\|u\|^2_{L^2(\Sigma,\dd S \dd t)},
\end{aligned}
 \end{equation}
where we have used \eqref{E:crossT01} to get rid of the crossed terms. We also have, using an integration by part: 
\begin{multline}
\label{E:ippcrossedT}
\int_{\Sigma} v(s) \psi'(t)\partial_{t}w(s,t) \dd S \dd t=\int_{S}v(s)\left(\Big[\psi'(t) w(s,t)\Big]_{t=0}^{t=\delta}-\int_{0}^{\delta}\psi''(t) w(s,t) \dd t\right) \dd S
\\
=\int_{S}v(s)\left(\psi'(\delta) w(s,\delta)-\psi'(0)w(s,0)+E^{\beta}\int_{0}^{\delta}\psi(t) w(s,t) \dd t \right) \dd S
\\
=\int_{S}\Big(\beta v(s) \psi(\delta) w(s,\delta)+\alpha v(s)\psi(0)w(s,0) \Big)\dd S,
\end{multline}
where we have used \eqref{E:crossT01} and the boundary condition for the eigenfunction $\psi$. Moreover,
by the definition of $\psi$ we have
\[
\int_0^\delta \psi'(t)^2 \dd t-\alpha \psi(0)^2-\beta \psi(\delta)^2=E^{\beta}.
\]
Inserting the last inequality and \eqref{E:ippcrossedT} into \eqref{E:DevLBI}, we arrive at
 \begin{equation}
 \label{E:DevLBII}
\begin{aligned}
 q_{\alpha}^{-}(u,u)=&\,(1-c_{0}\delta)\int_{S}g^{\rho\mu}\partial_{\rho}v\partial_{\mu}v \dd S+(1-c_{0}\delta)\int_{\Sigma}g^{\rho\mu}\partial_{\rho}w\partial_{\mu}w \dd S \dd t
 \\
&+\int_{\Sigma}|\partial_{t}w|^2 \dd S \dd t-  \int_{S}K(s)v(s)\psi(0)w(s,0) \dd S
\\
& - \int_{S}\big(\alpha+\dfrac{K}{2}\big) w(s,0)^2 \dd S-\int_{S}\dfrac{K}{2} v(s)^2 \psi(0)^2 \dd S\\
&-\beta \int_{S}w(s,\delta)^2 \dd S+E^{\beta}\|v\|^2_{L^2(S,\dd S)}-c_{1}\|u\|^2_{L^2(\Sigma,\dd S \dd t)}.
\end{aligned}
\end{equation}
  \begin{lemma}\label{lem65}
  There exist $R>0$ and $\alpha_{0}>0$ such that for all $\alpha\geq\alpha_{0}$ there holds
 \begin{multline}
 \label{E:lb1000par}
 \int_{\Sigma}|\partial_{t}w(s,t)|^2 \dd S \dd t-  \int_{S}K(s)v(s)\psi(0)w(s,0) \dd S\\
- \int_{S}\Big(\alpha+\dfrac{K(s)}{2}\Big)w(s,0)^2 \dd S  -\int_{S}\dfrac{K(s)}{2} v(s)^2\psi(0)^2 \dd S-\beta \int_{S}w(s,\delta)^2 \dd S
\\
\geq -\frac{\alpha^2}{2}\| w \|^2_{L^2(\Sigma,\dd S\dd t)}-\psi(0)^2\int_S\Big(\dfrac{K(s)}{2}+\dfrac{R}{\alpha}\Big) v(s)^2\dd S
  \end{multline}
	for all $u\in \cD(q^-_\alpha)$. Moreover, the constant $R$ depends only on $\|K\|_{\infty}$.
 \end{lemma}
 \begin{proof}
Denote by $J$ the term on the left-hand side of \eqref{E:lb1000par}. For any $\varepsilon>0$
we have
\[
\left|\int_{S}K(s)v(s)\psi(0)w(s,0) \dd S\right|
 \leq
\varepsilon \psi(0)^2 \int_{S}v(s)^2 \dd S+\dfrac{1}{4\varepsilon}\int_{S}K(s)^2 w(s,0)^2 \dd S,
\]
and there holds, for sufficiently small $\varepsilon$,
\begin{multline}
J \geq \int_{\Sigma}|\partial_{t}w(s,t)|^2 \dd S \dd t - \big(\alpha+\dfrac{B}{\varepsilon}\big)\int_{S} w(s,0)^2 \dd S
 \\
 -\psi(0)^2\int_{S}\Big(\dfrac{K(s)}{2}+\varepsilon\Big) v(s)^2 \dd S-\beta \int_{S}w(s,\delta)^2 \dd S,
  \label{eq-jj1}
\end{multline}
 with $B=\sup_{S}(K^2+|K|)$. Due to \eqref{E:wperppsi} and to the inequality \eqref{eq-fpsi} of Lemma \ref{lemn} we have
\[
\int_{0}^{\delta}\big|\partial_{t}w(s,t)\big|^2 \dd t - \alpha w(s,0)^2-\beta w(s,\delta)^2\geq 0, \quad s\in S.
\]
It follows that for any $\eta\in(0,1)$ we can estimate
\begin{multline*}
\int_{\Sigma}|\partial_{t}w(s,t)|^2 \dd S \dd t\\
 \geq \eta \int_{\Sigma}|\partial_{t}w(s,t)|^2 \dd S \dd t +(1-\eta) \alpha \int_{S}w(s,0)^2 \dd S +(1-\eta)\beta \int_{S}w(s,\delta)^2 \dd S,
\end{multline*}
and the substitution into \eqref{eq-jj1} gives
\begin{multline*}
J \geq \eta \int_{\Sigma}|\partial_{t}w(s,t)|^2 \dd S \dd t-\Big(\eta\alpha+\dfrac{B}{\varepsilon}\Big)\int_{S}w(s,0)^2 \dd S
-\eta\beta \int_{S} w(s,\delta)^2 \dd S\\
{}-\psi(0)^2\int_{S}\Big(\dfrac{K(s)}{2}+\varepsilon\Big) v(s)^2 \dd S.
\end{multline*}
Therefore, choosing $\varepsilon=R/\alpha$ with $R>0$ and then using Lemma \ref{lemsob} we obtain, for any $\ell \in (0,\delta)$,
\[
\begin{aligned}
J\geq&\, \eta \int_{\Sigma}|\partial_{t}w(s,t)|^2 \dd S \dd t-\Big(\eta+\dfrac{B}{R} \Big)\alpha \bigg(\ell \int_{\Sigma}|\partial_{t}w(s,t)|^2\dd S \dd t
+\dfrac{2}{\ell}\int_{\Sigma}|w(s,t)|^2\dd S \dd t\bigg)
\\
&-\eta \beta \left(\ell \int_{\Sigma}|\partial_{t}w(s,t)|^2\dd S \dd t +\dfrac{2}{\ell}\int_{\Sigma}w(s,t)^2\dd S \dd t\right)
-\psi(0)^2\int_{S}\Big(\dfrac{K(s)}{2}+\dfrac{R}{\alpha}\Big) v(s)^2 \dd S
\\
=&\,
\Big[\eta-\ell\alpha\Big(\eta+\dfrac{B}{R}\Big)-\ell\beta\eta\Big] \int_{\Sigma}|\partial_{t}w(s,t)|^2 \dd S \dd t-\dfrac{2}{\ell}\Big(\eta\alpha+\dfrac{B\alpha}{R}+\eta\beta\Big)\int_{\Sigma}w(s,t)^2 \dd S \dd t 
\\
&-\psi(0)^2\int_{S}\Big(\dfrac{K(s)}{2}+\dfrac{R}{\alpha}\Big) v(s)^2 \dd S.
\end{aligned} 
\]
Choose $\ell=\rho(\alpha+\beta)^{-1}$ with $\rho\in(0,1/2)$ and $R \geq B (1-\rho)^{-1}$, then
the choice
\[
\eta=\frac{\ell B \alpha}{R\big(1-\ell (\alpha+\beta)\big)}=\frac{\rho B\alpha}{(\alpha+\beta)R(1-\rho)}\in (0,1),
\]
implies
\[
\eta-\ell\alpha\Big(\eta+\dfrac{B}{R}\Big)-\ell\beta\eta=0,
\]
and
\begin{equation}
\label{E:lastlwJ}
J\geq -\left(\frac{2B\alpha(\alpha+\beta)}{R(1-\rho)}+\frac{2B\alpha(\alpha+\beta)}{\rho R} \right) \|w\|^2_{L^2(\Sigma,\dd S\dd t)}
-\psi(0)^2\int_{S}\Big(\dfrac{K(s)}{2}+\dfrac{R}{\alpha}\Big) v(s)^2 \dd S.
\end{equation}
As $R$ can be taken arbirary large, we may choose it in order to have
\[
R > 4B\left( \frac{1}{1-\rho}+\frac{1}{\rho}\right),
\]
then there exists $\alpha_{0}>0$ such that for $\alpha\ge \alpha_0$ we have
\[
\left(\frac{2B\alpha(\alpha+\beta)}{R(1-\nu)}+\frac{2B\alpha(\alpha+\beta)}{R\nu} \right)<\frac{\alpha^2}{2},
\]
which gives the result.
 \end{proof}

Substituting the result of Lemma~\ref{lem65} into \eqref{E:DevLBII} 
and using the equality
\[
\|u\|^2_{L^2(\Sigma,\dd S \dd t)}=\|v\|^2_{L^2(S,\dd S)}+\|w\|^2_{L^2(\Sigma,\dd S \dd t)}
\]
we deduce
 \begin{multline}
 \label{E:DevLBI2}
 q_{\alpha}^{-}(u,u) -E^{\beta}\|u\|^2_{L^2(\Sigma,\dd S \dd t)}
 \geq 
 (1-c_{0}\delta)\int_{S}g^{\rho\mu}\partial_{\rho} v(s)\partial_{\mu} v(s) \dd S
 \\
 -\psi(0)^2\int_S \Big(\dfrac{K(s)}{2}+\dfrac{R}{\alpha}\Big) v(s)^2 \dd S-
\Big(E^{\beta}+\dfrac{\alpha^2}{2}\Big)\|w\|_{L^2(\Sigma,\dd S \dd t)}-c_{1}\|u\|^2_{L^2(\Sigma,\dd S \dd t)}.
 \end{multline}
 We choose
$\delta= \alpha^{-\kappa}$, $\kappa\in [2/3,1)$,
then Lemma \ref{lemn} provides
\[
E^{\beta}=-\alpha^2+o(1), \quad \psi(0)^2=2\alpha+o(1).
\]
Using rough estimates, we deduce, as $\alpha\to+\infty$:
\begin{multline*}
q_{\alpha}^{-}(u,u) -E^{\beta}\|u\|^2_{L^2(\Sigma,\dd S \dd t)}
 \geq 
 (1-c_{0}\delta)\int_{\Sigma}g^{\rho\mu}\partial_{\rho} v\partial_{\mu} v\, \dd S
 \\
 -\int_{\Sigma}(\alpha K+2R+1)v(s)^2 \dd S+\frac{\alpha^2}{4}\|w\|^2_{L^2(\Sigma,\dd S \dd t)}-c_{1}\|u\|^2_{L^2(\Sigma,\dd S \dd t)}.
\end{multline*}
Following the arguments of Section \ref{SS:lb2}, we get 
\[
E_{j}(Q_{\alpha}^{-})-E^{\beta} \geq E_{j}\big(-(1-c_{0}\delta)\Delta_{S}-\alpha K\big)+\cO(1).
\]
Finally, we get the desired lower bound of Theorem \ref{thm2} by using \eqref{E:lbvpopC3}. 

\section{Proof of Corollaries \ref{cor1} and \ref{cor1a}}\label{secor}

In this section, we assume that $\Omega$ is $C^2$-admissible. 
Then the function $K$ is bounded, and an easy adaptation of~\cite[Proposition 1]{FK}
to the non-euclidean setting gives the following:
\begin{lemma}\label{lemfk}
For any fixed $j\in\NN$ there holds
\begin{equation}
\label{E:asymptgross}
E_{j}(-\Delta_{S}- \alpha K)=-K_{\max}\alpha+o(\alpha), \quad \alpha\to+\infty.
\end{equation}
\end{lemma}

\begin{proof}[\bf Proof of Corollary~\ref{cor1}]
It is sufficient to substitute the estimate \eqref{E:asymptgross} into the asymptotics \eqref{eq-th1}
of Theorem~\ref{thm1}.
\end{proof}

To prove Corollary \ref{cor1a} we need a rough estimate for the essential spectrum of $Q^\Omega_\alpha$. Recall that for a self-adjoint operator $Q$, we have denoted by $E(Q)$ the infimum of its essential spectrum. Then there holds:

\begin{lemma}\label{lemess}
Assume that $\partial\Omega$ is non-compact and denote
$K_\infty:=\limsup_{s\to\infty}K(s)$, then
$E(Q^\Omega_\alpha)\ge -\alpha^2- K_\infty \alpha+o(\alpha)$
 for large $\alpha$.
\end{lemma}

\begin{proof}
Lets us modify a bit the construction of Subsection \ref{dtnbra}.
By assumption, for any $K_0> K_\infty$
there exists a compact domain
$S_0\subset S$ such that $K(s)\le K_0$ for
$s\in S\setminus S_0$. Set $S_1:=S\setminus \overline{S_0}$.
 Now let $q'_\alpha$ denote
the quadratic form given by the same expression as $q^\Omega_\alpha$
but acting on the domain
$\cD(q'_\alpha):=H^1(\Omega^0_\delta)\oplus H^{1}( \Omega^1_\delta )\oplus H^{1}(\Theta_\delta)$
with
\[
\Omega^0_\delta:=\Phi\big(S_0,(0,\delta)\big),
\quad
\Omega^1_\delta:=\Phi\big(S_1,(0,\delta)\big),
\quad
\Theta_\delta:=\Omega\setminus \overline{\Omega^0_\delta\cup \Omega^1_\delta}
\]
and $\delta$ is sufficiently small. Let $Q'_\alpha$ be the self-adjoint operator associated with $q'_\alpha$
and acting in $L^2(\Omega)$. Due to the form inequality $Q_\alpha^\Omega\ge Q'_\alpha$
we have $E(Q^\Omega_\alpha)\ge E(Q'_\alpha)$.
On the other hand, one represents $Q'_\alpha=Q^0_\alpha \oplus Q^1_\alpha \oplus (-\Delta)_{\Theta_\delta}^N$,
where $Q^j_\alpha$, $j\in\{0,1\}$, is the self-adjoint operator in $L^2(\Omega^j_\delta)$
generated by the quadratic form
\[
q^j_\alpha(u,u)=\int_{\Omega^j_\delta}|\nabla u|^2\dd x-\alpha \int_{\partial \Omega^j_\delta\cap S} u^2\dd S,
\quad
\cD(q^j_\alpha)=H^1(\Omega^j_\delta),
\]
and $(-\Delta)_{\Theta_\delta}^N$ is the Neumann Laplacian in $\Theta_\delta$.
Note that the domain $\Omega^0_\delta$ is bounded, hence,
the operator $Q^0_\alpha$ has an empty essential spectrum.
It follows that
\begin{equation*}
E(Q^\Omega_\alpha) \ge 
E(Q'_\alpha)
=\min\big(
E(Q^1_\alpha),
E((-\Delta)_{\Theta_\delta}^N)\big)
\ge\min\big(E(Q^1_\alpha),0\big).
\end{equation*}
On the other hand, the analysis of Section~\ref{SS:lb1} can be applied
to the operator $Q^1_\alpha$. In particular, the choice \eqref{eq-dba} for $\delta$ gives
\[
E_1(Q^1_\alpha)\ge -\alpha^2+E_1(-\Delta_{S_1}^N-\alpha K)+\cO(\log\alpha),
\]
where $-\Delta_{S^1}^N$ is the Neumann realization of the positive Laplace-Beltrami
operator on $S_1$. As $K\le K_0$ in $S_1$, we have
\[
E(Q^1_\alpha)\ge
E_1(Q^1_\alpha)\ge -\alpha^2-K_0 \alpha+\cO(\log\alpha)
\]
and, subsequently, $E(Q^\Omega_\alpha)\ge -\alpha^2-K_0 \alpha+\cO(\log\alpha)$.
As $K_0>K_\infty$ is arbitrary, the result follows.
\end{proof}

\begin{proof}[\bf Proof of Corollary~\ref{cor1a}]
Let $N\in\NN$ be fixed. Due to Corollary \ref{cor1} and Lemma  \ref{lemess}
for large $\alpha$ we have
$E_N(Q^\Omega_\alpha)< E(Q^\Omega_\alpha)$, and
$E_N(Q^\Omega_\alpha)$ is the $N^{\mbox{th}}$ eigenvalue of $Q^\Omega_\alpha$
due to the min-max principle.
\end{proof}

\section{Analysis of the reduced operator on the boundary}
\label{S:semiclassical}
In this section we gather various standard estimates on the low-lying eigenvalues of 
$-\Delta_{S}- \alpha K$, depending on the hypotheses on the dimension $\nu$ and on $K$. In this section,
for the cas of an unbounded $\partial\Omega$ we assume that the assumption~\eqref{E:H3} holds.

Now note that by setting $h=\alpha^{-1/2}$ and $V=-K$, the operator $-\Delta_{S}- \alpha K$ writes as 
\[
-h^{-2}\left(-h^2\Delta_{S}+V \right)
\]
and enters naturally the framework of Schr\"odinger operators in the semi-classical limit $h\to0$.
The assumption \eqref{E:H3} writes now
$$\liminf_{s\to\infty}V(s) > \inf_{s\in S} V(s),$$
and ensures that the asymptotics of the low-lying eigenvalues of the reduced operator can be determined by the behavior of $V$ near its minima (that are the maxima of $K$), under suitable hypotheses.

\begin{rem}
\label{R:measuremax}
 Assume that the measure of the set $K^{-1}(\{K_{\max}\})$ is 0. Then the word-by-word
 adaptation of~\cite[Lemma 3.2]{BHV}
to the non-euclidean setting gives, for any fixed $j\in\NN$,
\begin{equation}
  \label{eq-ejej}
E_{j}\big(-\Delta_{S}+ \alpha (K_{\max}-K)\big)\to+\infty \text{ as }\alpha\to+\infty
\end{equation}
If, in addition, $\Omega$ is $C^3$-admissible, then \eqref{eq-th1a} can be decomposed as
\[
E_j(Q^\Omega_\alpha)=-\alpha^2 - K_\text{max} \alpha
+ 
E_{j}\big(-\Delta_{S}+ \alpha (K_{\max}-K)\big)
+\cO(1),
\]
and the term $E_{j}\big(-\Delta_{S}+ \alpha (K_{\max}-K)\big)$
has a lower order with respect to $\alpha$, see Lemma~\ref{lemfk}, but
is large with respect to the remainder $\cO(1)$, see \eqref{eq-ejej},
and hence provides a refinement
with respect to the first order asymptotics~\eqref{eq-hhh}.
\end{rem}
The aim is now to describe more precise asymptotics on $-\Delta_{S}- \alpha K$, in order to see the possible gap between eigenvalues, in particular we want to compare $E_{j}(-\Delta_{S}- \alpha K)+K_{\max}\alpha$ to the remainders in Theorems \ref{thm1}
and \ref{thm2}. The most commonly studied
case is when the maxima of $K$ 	are non-degenerate, see \cite[Theorem 5.1]{Si83} or \cite{HS84}:
\begin{prop}
\label{P:nondegenerate}
Assume that the boundary of $\Omega$ is $C^{5}$ and, if non-compact,
satisfies \eqref{E:H3}. Furthemore, assume that the function $K$ admits a unique maximum at $s_{0}\in S$
and that the Hessian of $(-K)$ at $s_{0}$ is positive-definite. Denote by $\mu_{k}$ the eigenvalues
of the Hessian and 
\begin{equation}
   \label{eq-eee}
\cE:=\Big\{ \sum_{k=1}^{\nu-1}\sqrt{\frac{\mu_{k}}{2}}(2n_{k}-1),\quad n_{k}\in \NN\Big\},
\end{equation}
then for each fixed $j\in\NN$ there holds:
\[
E_{j}(-\Delta_{S}- \alpha K)=-K_{\max}\alpha+e_{j}\alpha^{1/2}+\cO(\alpha^{1/4}) \text{ as } \alpha\to+\infty,
\]
where $e_{j}$ is the $j^{\mbox{th}}$ element of $\cE$, counted with multiplicity. Moreover, if $\Omega$ is $C^{6}$, and if 
$e_{j}$ is of multiplicity one, the remainder can be replaced by $O(1)$.
\end{prop}
By combining Proposition~\ref{P:nondegenerate} with Theorem~\ref{thm1} we obtain Corollary~\ref{C:nondegenerate}. Remark
that for $\nu=2$ one is reduced to
\[
\cE=\Big\{ \sqrt{\frac{-K''(s_{0})}{2}}(2n-1),\quad n\geq1\Big\}
\]
and all the elements are of multiplicity one. Therefore, by
combining Theorem \ref{thm1} and Proposition \ref{P:nondegenerate}, we recover the first terms
of the asymptotic expansion \eqref{eq-HK}, see \cite[Theorem 1.1]{HK}.

Other cases of extrema are harder to handle, due to the different notions of degeneracy for the maxima of $K$, and to the possible interactions with the metric near the maxima. However, in the case $\nu=2$, we have the following: 
\begin{prop} 
\label{P:degenerate}
Let $\nu=2$ and $\Omega$ be $C^{2p+3}$-admissible with $p\geq2$ and,
if $\partial\Omega$ is non-compact, such that
the assumption \eqref{E:H3} is satisfied. Furthermore, assume that
the curvature $K$ admits a unique global maximum at $s_{0}$ with
\[
K(s)=K(s_{0})-C_{p}(s-s_{0})^{2p}+\cO\big((s-s_{0})^{2p+1}\big), \quad s\to s_0,
\]
where $C_p>0$ and  $s$ is an arc-length of the connected component $\Gamma$ of the boundary at which $K$ takes the maximal value.
Then we have the following expansion
\[
E_{j}(-\Delta_{S}- \alpha K)=-K_{\max}\alpha+e_{j}\alpha^{\frac{1}{p+1}}+\cO(\alpha^{\frac{1}{2(p+1)}}),
\]
 where $e_{j}$ is the $j^{\mbox{th}}$ eigenvalue of the operator $-\partial_{s}^2+C_{p}s^{2p}$ in $L^2(\RR)$.
 Moreover, if $\partial \Omega$ is $C^{2p+4}$, then the remainder can be replaced by $\cO(1)$.
 \end{prop}
 \begin{proof}
Since we are not interested in exponentially small terms, it suffices by standard arguments to reduce the analysis to a neighborhood of the minimizer in $\Gamma$, denoted by $\Gamma_{0}$, with Dirichlet boundary conditions at the ends, see \cite{HS84}.
Let $\gamma :\RR/|\Gamma_{0}| \ZZ \to \Gamma_{0}$ be an arc-length parametrization of $\Gamma_{0}$. Since the parametrization is normalized and the metrics in local coordinates is $g=\|\gamma'\|$, we only have to consider $-\Delta-\alpha K$ on the interval $(s_{0}-\eta,s_{0}+\eta)$, with $\eta>0$ fixed, and Dirichlet boundary condition. The following asymptotics is then a simple consequence of \cite[Theorem 2.1]{MarRou88} applied with the semi-classical parameter $h=\alpha^{-1/2}$:
\[
E_{j}(-\Delta_{S}- \alpha K)=\alpha^{\frac{1}{p+1}}\left(e_{j}+\sum_{k\geq1} \beta_{j,k}\alpha^{-\frac{k}{2(p+1)}}\right), \quad \beta_{j,k}\in \RR.
\]
If $\partial\Omega$ is $C^{2p+4}$, then the curvature is $C^{2p+2}$, and we have the Taylor expansion
\[
K(s)=K(s_{0})-C_{p}(s-s_{0})^{2p}+C_{p}'(s-s_{0})^{2p+1}+\cO\Big((s-s_{0}\Big)^{2p+2}), \quad C_{p}'\in \RR.
\]
Then, by combining the simplicity of the eigenvalues $(e_{j})_{j\geq1}$, the parity of the eigenvectors of $-\partial_{s}^2+C_{p}s^{2p}$, and the oddness of the remainder $C_{p}'(s-s_{0})^{2p+1}$ in the asymptotic expansion of $K$, it is standard to show that $\beta_{j,2}=0$ for all $j\geq1$, see for example \cite[Theorem 4.23]{DiSj99} for the case $p=1$.
 \end{proof}
The combination of Proposition~\ref{P:degenerate} with Theorem~\ref{thm1} gives Corollary~\ref{C:degenerate}.

\begin{rem}
The above statements can be adapted easily to the case where $K$ has several maxima by using the principle that ``each well creates its own
series of eigenvalues''. 
\end{rem}

\begin{cor}
\label{C:simplicity}
Let $j \in \NN$, and assume \emph{one} of the two following:
\begin{itemize}
\item The hypotheses of Proposition \ref{P:nondegenerate} hold, and $e_{j}$ is of multiplicity 1 in the set $\cE$.
\item The hypotheses of Proposition \ref{P:degenerate} hold. 
\end{itemize}
Then, for $\alpha$ large enough, $E_{j}(Q_{\alpha}^{\Omega})$ is a simple eigenvalue.
\end{cor}


When we are not in the hypotheses of Remark \ref{R:measuremax}, few results exist on the asymptotics of the first eigenvalues. For example, we can show
\begin{prop}\label{prop86}
Assume that the interior of $K^{-1}(\{K_{\max}\})$ is not empty. Then, for any fixed $j\in\NN$,
\[
E_{j}(-\Delta_{S}- \alpha K)=-\alpha K_{\max}+\cO(1) \text{ as } \alpha\to+\infty.
\]
\end{prop}
\begin{proof}
Denote by $\omega\subset \partial\Omega$ an open subset of the interior of $K^{-1}\{K_{\max}\}$ with a smooth
boundary. Introduce $-\Delta^{D}_{\omega}$, the Laplace-Beltrami operator in $\omega$ with the Dirichlet boundary condition. This operator has compact resolvent and we denote by $E_{j}^{D}(\omega)$, $j\in\NN$, its eigenvalues,
and by $u_{j}$ associated normalized eigenfunctions.
Denote by $U_j$ the extensions of $u_{j}$ to $\partial\Omega$ by zero, then
\[
\int_S g^{\rho\mu} \partial_\rho U_j\partial_\mu U_j \dd S
-\alpha
\int_S K  U_j^2 \dd S = -\alpha K_{\max}+E_{j}^{D}(\omega).
\]
As $U_j$ are mutually ortohogonal in $L^2(S,\dd S)$,
we deduce from the min-max principle that
$E_{j}(-\Delta_{S}- \alpha K) \leq -\alpha K_{\max}+E_{j}^{D}(\omega)$,
and the sought estimate follows.
\end{proof}
In particular, in the situation of Proposition~\ref{prop86}
Theorems \ref{thm1} and \eqref{thm2}
does not provide the gap between the eigenvalues of $Q_{\alpha}^{\Omega}$
as $\alpha\to+\infty$.

We remark that a particular case of a piecewise constant curvature was recently studied in~\cite{kp14}, and
the eigenvalue gaps appear to have finite limits.

\section{Periodic case}\label{secper}

The preceding analysis can also be applied to periodic problems.
Namely, assume that there exist linearly independent vectors
$a_1, \dots, a_m$, $m\le \nu$, such that
$\Omega$ is invariant under the shifts $x\mapsto x+a_j$, $j\in\{1,\dots,m\}$,
and that the quotient (elementary cell) $\omega:=\Omega/(\ZZ a_1+\dots+\ZZ a_m)$ is compact, and then the
quotient surface $\sigma:=S/(\ZZ a_1+\dots+\ZZ a_m)$ is also compact.
Such a situation is covered by the Floquet theory \cite{kbook}.
Namely, for $\theta=(\theta_1,\dots,\theta_m)\in \TT^m$,
$\TT:=\RR/ 2\pi \ZZ$ denote by
$Q^\Omega_\alpha(\theta)$ the self-adjoint operator
generated by the quadratic form
\begin{gather*}
q^{\Omega,\theta}_\alpha(u,u):=\int_\omega |\nabla u|^2\dd x
-\alpha\int_\sigma u^2 \dd S,\\
\cD(q^{\Omega,\theta}_\alpha)
=H^1_\theta(\omega):=\Big\{u\in H^1_\mathrm{loc}(\Omega): u(\cdot+a_j)=e^{i\theta_j}u(\cdot), \quad j=1,\dots,m\Big\}.
\end{gather*}
It can be easily checked that the operators $Q^\Omega(\theta)$ are with compact resolvents,
and it is a standard fact of the Floquet theory that for each fixed $j\in \NN$
the so-called band function
\[
\TT^m\ni \theta \mapsto E_j(\theta,\alpha):=E_j\big(Q^\Omega_\alpha(\theta)\big)
\]
is continuous, and that
\[
\spec Q^\Omega_\alpha:= \bigcup_{j\in \NN} B_j (\alpha),
\quad
B_j(\alpha):=\big\{
E_j(\theta,\alpha):\, \theta\in \TT^m
\big\}.
\]
The segment $B_j(\alpha)$ is usually called the $j^{\mbox{th}}$ spectral band of $Q^\Omega_\alpha$.

An analogous representation of the spectrum holds for the reduced operator
$-\Delta_S-\alpha K$. Namely, denote by $T_\alpha(\theta)$ the self-adjoint operator
acting in $L^2(\sigma)$ associated with the quadratic form
\begin{gather*}
t_\alpha^\theta(v,v):= \int_\sigma g^{\rho\mu}\partial_\rho v \partial_\mu v \,\dd S
-\alpha \int_\sigma K v^2 \dd S,\\
\cD(t^\theta_\alpha)=H^1_\theta(\sigma):=
\Big\{u\in H^1_\mathrm{loc}(S): u(\cdot+a_j)=e^{i\theta_j}u(\cdot), \quad j=1,\dots,m\Big\}.
\end{gather*}
Again, one checks that $T_\alpha(\theta)$ have compact resolvents
and the band functions
\[
\TT^m\ni \theta\mapsto \varepsilon_j(\theta,\alpha):=E_j\big(T_\alpha(\theta)\big)
\]
are continuous and
\[
\spec (-\Delta_S-\alpha K):= \bigcup_{j\in \NN} \beta_j (\alpha),
\quad
\beta_j(\alpha):=\big\{
\varepsilon_j(\theta,\alpha):\, \theta\in \TT^m
\big\},
\]
and the segment $\beta_j(\alpha)$ will be called the $j^{\mbox{th}}$ spectral band of $-\Delta_S-\alpha K$.

One can easily see that the proofs of Theorems~\ref{thm1} and  \ref{thm2} also work
for the operators $Q^\Omega_\alpha(\theta)$, which gives the following results:
\begin{theorem}\label{thper}
For any fixed $j\in\NN$ there holds, as $\alpha\to +\infty$,
\[
E_j(\theta,\alpha)=-\alpha^2+\varepsilon_j(\theta,\alpha) +R_j(\theta,\alpha), \quad \theta\in \TT,
\]
where $R_j(\theta,\alpha)=\cO(\log \alpha)$ if $\Omega$ is $C^2$ and
$R_j(\theta,\alpha)=\cO(1)$ if $\Omega$ is $C^3$, and
the remainder estimate is uniform in $\theta\in \TT$. 
\end{theorem}

\begin{cor}
If $j\in\NN$ is fixed and $\alpha\to+\infty$, then the $j^{\mbox{th}}$ spectral band of $Q^\Omega_\alpha+\alpha^2$
and the $j^{\mbox{th}}$ spectral band of $-\Delta_S-\alpha K$ are located
in a $\cO\big(R(\alpha)\big)$-neighborhood of each other, where
$R(\alpha)=\log \alpha$ for the $C^2$-admissible case
and $R(\alpha)=1$ for the $C^3$-admissible one.
\end{cor}

The result of Theorem \ref{thper} can be used to study some spectral properties
specific for periodic operators. Recall that  a non-empty interval $(a,b)\subset \RR$ is called a (spectral) gap of a self-adjoint operator $A$
if $(a,b)\cap \spec A=\emptyset$ but $a,b\in \spec A$. The existence of spectral gaps
is one of the principal questions in the spectral theory of periodic operators,
cf.~\cite{bp,hepo,khrab,naz}. In view of Theorem \ref{thper},
the existence of sufficiently large gaps for the reduced operator
$-\Delta -\alpha K$ (i.e. having the length of order $\alpha^\kappa$ with some $\kappa>0$)
implies the existence of gaps
for the Robin Laplacian $Q^\Omega_\alpha$, and the reduced operator was studied
in numerous preceding works, cf.~\cite{out,shubin}. For, example the semiclassical
analysis of periodic operators of the form $-h\Delta_S+V/h$ carried out
in \cite[Theorem~1.1]{shubin} gives the following result:

\begin{cor}
Assume that $\Omega$ is $C^\infty$ and periodic as described above.
Furthermore, assume that the function $\sigma\ni s\mapsto K(s)$ admits a unique maximum at $s_{0}$,
and that the Hessian of $(-K)$ at $s_{0}$ is positive-definite. Let $\mu_j$ be the eigenvalues of the Hessian
and  the numbers $e_{j}$ be defined as in Proposition \ref{P:nondegenerate},
then
\begin{enumerate}
\item
for each $j\in \NN$ there exists $C>0$ such that 
\[
\spec (Q^\Omega_\alpha+\alpha^2 + K_{\max}\alpha) \cap \big[e_j \alpha^{1/2} - C \alpha^{2/5}, e_j \alpha^{1/2} + C \alpha^{2/5}\big]\ne \emptyset
\]
for large $\alpha$, and
\item
for each $C_1>0$ there exist $C_2,C_3>0$ such that
\begin{multline*}
\spec (Q^\Omega_\alpha+\alpha^2 + K_{\max} \alpha) \cap \big[-C_1\alpha^{1/2},C_1 \alpha^{1/2}\big]\\
 \subset
 \bigcup_{e_j\le C_3} \big[e_j \alpha^{1/2} - C_2 \alpha^{2/5},e_j \alpha^{1/2} + C_2 \alpha^{2/5}\big]
\end{multline*}
as $\alpha\to+\infty$.
\end{enumerate}
In particular, for any $N\in \NN$ there exists $\alpha_N>0$ such that
the operator $Q^\Omega_\alpha$ has at least $N$ gaps for $\alpha>\alpha_N$.
\end{cor}
The localization of the spectrum given in the preceding corollary is not expected
to be optimal for periodic domains. Furthermore, it would be interesting to understand some
questions related to the location of the extrema of the band functions, cf.~\cite{bp}.
 We hope to analyze the periodic case in greater detail in subsequent works.

\end{document}